\documentclass[12pt, reqno]{amsart}

\usepackage{fullpage}
\usepackage{amsmath,amssymb,amsthm,mathtools}
\usepackage{booktabs}
\usepackage{hyperref}
\usepackage{colonequals}

\hypersetup{colorlinks=true,linkcolor=blue,citecolor=blue,urlcolor=blue}

\newtheorem{theorem}{Theorem}
\newtheorem{lemma}[theorem]{Lemma}
\newtheorem{remark}[theorem]{Remark}

\newtheorem{cor}[theorem]{Corollary}
\newtheorem{question}[theorem]{Question}
\newtheorem{answer}[theorem]{Answer}

\theoremstyle{definition}
\newtheorem{algorithm}[theorem]{Algorithm}
\newtheorem{example}[theorem]{Example}
\newtheorem{definition}[theorem]{Definition}
\newtheorem{conj}[theorem]{Conjecture}
\newtheorem{urlpage}{GitHub Link}

\numberwithin{theorem}{section}

\newcommand{\R}{\mathbb R}
\newcommand{\C}{\mathbb C}
\newcommand{\D}{\mathbb D}
\newcommand{\E}{\mathbb E}

\newcommand{\dd}{\,d}
\newcommand{\abs}[1]{\left|#1\right|}
\newcommand{\norm}[1]{\left\|#1\right\|}
\newcommand{\ip}[2]{\left\langle #1,#2\right\rangle}
\newcommand{\sgn}{\operatorname{sign}}
\newcommand{\Repart}{\Re}

\newcommand{\KG}{K_G}

\newcommand{\erf}{\operatorname{erf}}
\renewcommand{\Subset}{\subset}

\title{An Upper Bound on Grothendieck's Constant}
\author{Steven Heilman}
\date{\today}

\thanks{
Email: stevenmheilman@gmail.com\\
Supported by NSF Grant CCF AF 2448108.
Department of Mathematics, University of Southern California, Los Angeles, CA 90089}

\begin{document}

\begin{abstract}
We show that Grothendieck's real constant $K_G$ can be upper bounded by projecting vectors onto a random plane through the origin and thresholding a degree five Hermite polynomial.  This resolves a conjecture of Braverman-Makarychev-Makarychev-Naor from 2011, who required an extra randomization step in their rounding scheme and proved $K_G<\frac{\pi}{2\log(1+\sqrt{2})}-10^{-500}$.  As a corollary of our result, we prove the bound $K_G<\frac{\pi}{2\log(1+\sqrt{2})}-10^{-217}$ by thresholding degree three Hermite polynomials in the plane.  We finally give a rigorous computer-assisted proof that $K_G<\frac{\pi}{2\log(1+\sqrt{2})}-10^{-5}$ using interval arithmetic and degree three Hermite polynomial thresholding.
\end{abstract}
%These new $K_G$ upper bounds improve on the previous upper bound from $2011$.

\maketitle

\section{Introduction}

Grothendieck's real constant $K_{G}$ is the infimum over all $K\in(0,\infty)$ such that, for all positive integers $m,n$ and for every real $m\times n$ matrix $(a_{ij})$, we have
\begin{equation}\label{groeq}
\max_{\substack{x_{1},\ldots,x_{m}\\y_{1},\ldots,y_{n}}\in S^{m+n-1}}\sum_{i=1}^{m}\sum_{j=1}^{n}a_{ij}\langle x_{i},y_{j}\rangle
\leq K\cdot
\max_{\substack{\epsilon_{1},\ldots,\epsilon_{m}\\\delta_{1},\ldots,\delta_{n}}\in\{-1,1\}}\sum_{i=1}^{m}\sum_{j=1}^{n}a_{ij}\epsilon_{i}\delta_{j},
\end{equation}
where $\langle x,y\rangle=\sum_{i=1}^{d}x_{i}y_{i}$ for all $x,y\in\R^{d}$ and $S^{d-1}=\{x\in\R^{d}\colon\langle x,x\rangle=1\}$ for any $d\geq1$.

Inequality \eqref{groeq} was originally stated as an inequality of two different tensor norms \cite{groth53} (see also \cite[Section 3]{pisier12}), though the discretized formulation \eqref{groeq} was proven in \cite{linden77}.

Determining the exact value of $K_{G}$ remains a significant open problem since it was first posited in \cite{groth53}.  We can rephrase the problem of finding the constant $K_{G}$ as: what is the ``best'' way to ``round'' the vectors $x_{1},\ldots,x_{m},y_{1},\ldots,y_{n}$ to $\pm1$?  There are many good references on \eqref{groeq} and its
connections to combinatorics, functional analysis, Banach space theory,
operator algebras, the Connes embedding problem, theoretical computer science,
and quantum mechanics; see, for example, \cite{pisier12,khot12}.  Note also that Grothendieck's constant is the maximal quantum violation in Bell's inequality from quantum mechanics \cite{tsirelson87}.  Also, assuming the Unique Games Conjecture \cite{khot02}, it is NP-hard to approximate the right side of \eqref{groeq} within any constant smaller than Grothendieck's constant $K_{G}$ \cite{ragh09}. 

% We briefly mention some interpretations of Grothendieck's constant in quantum information and theoretical computer science:
% \begin{itemize}
% \item Grothendieck's constant is the maximal quantum violation in Bell's inequality from quantum mechanics \cite{tsirelson87}.  That is, Grothendieck's inequality is a reformulated version of Bell's inequality.
% \item Assuming the Unique Games Conjecture \cite{khot02}, it is NP-hard to approximate the right side of \eqref{groeq} within any constant smaller than Grothendieck's constant $K_{G}$ \cite{ragh09}.  The left side of \eqref{groeq} is a semidefinite program which can be computed efficiently, while the right side is an integer program.  So, \eqref{groeq} itself efficiently approximates the right side of \eqref{groeq} using its left side, within a constant factor $K_{G}$ \cite{alon04}.  The cut norm and the MAX-CUT problem are special cases of the right side of \eqref{groeq} \cite{alon04}.
% %\item Inequality \eqref{groeq} implies a polynomial time approximation algorithm for the cut norm problem \cite{alon04}, which is the best possible polynomial time approximation assuming the Unique Games Conjecture \cite{ragh09}.  
% \end{itemize}
% %\cite{braverman13}
% % pisier survey

%\subsection{Upper Bounds on \texorpdfstring{$K_{G}$}{KG}}
Grothendieck \cite{groth53} originally proved that $K_{G}\leq\sinh(\pi/2)\approx2.3013$.  In \cite{krivine77}, Krivine improved this bound to $K_{G}\leq\frac{\pi}{2\log(1+\sqrt{2})}\approx 1.78221397819$ (see Example \ref{sinex} below), and it was generally believed that this inequality should be an equality \cite{konig01}.  However, it was then shown in \cite{braverman13} that there is a $c'>0$ such that 
\begin{equation}\label{bmmneq}
K_{G}<\frac{\pi}{2\log(1+\sqrt{2})} - c'.
\end{equation}
An effective $c'$ was proven but not specified in \cite{braverman13}; an inspection of the argument seems to give $c'=10^{-500}$.  Krivine's argument \cite{krivine77} shows that, after ``preprocessing'' the vectors $x_{1},\ldots,x_{m},y_{1},\ldots,y_{n}$ by nonlinearly mapping them to a different Hilbert space (Fock space), we can then map those vectors to $\pm1$ by projecting them onto a Gaussian random vector, and taking the sign of this projected value.  (The nonlinear preprocessing removes the nonlinearity that appears after projecting onto the Gaussian.)  The argument of \cite{braverman13} instead projects the preprocessed vectors onto a random plane through the origin, and then (with probability $0<p<1$) applies a perturbation of the sign function on this two-dimensional plane to ``round'' the vectors to $\pm1$ (and with probability $1-p$ applies the sign function).  This perturbation uses a degree 5 Hermite polynomial.

Also, rounding schemes of this form (projecting onto a random hyperplane through the origin and then thresholding, possibly with additional randomness) can obtain arbitrarily good approximations of $K_{G}$ \cite{regev14}.  In other words, finding the exact value of $K_{G}$ reduces to finding the best ``rounding scheme'' for vectors $x_{1},\ldots,x_{m},y_{1},\ldots,y_{n}$ in \eqref{groeq} (where the rounding scheme itself might need additional randomness).

The best known lower bounds for $K_G$ were recently shown in \cite{heilman26,jones26} to be around $1.6769$.  In particular, \cite{jones26} showed a $10^{-12}$ improvement over the bound from \cite{davie84,reeds93}.

\subsection{Krivine Rounding Schemes}

We now describe the rounding scheme we will use to prove the best known constant in Grothendieck's inequality.

\begin{definition}[\textbf{Krivine rounding scheme}, \cite{braverman13}]\label{kdef}
Fix \(k \in \mathbb{N}\) and let \(f,g\colon\mathbb{R}^k \to \{-1,1\}\) be odd measurable functions. Let
\(G_1,G_2 \in \mathbb{R}^k\) be independent standard Gaussian random vectors, so that $G_1$ has density
$
x \mapsto (2\pi)^{-k/2}e^{-\|x\|_2^2/2}
$.
For any $t\in(-1,1)$, define the (complex correlation parameter) noise stability
\begin{equation}\label{one8}
\begin{aligned}
H_{f,g}(t)
&\colonequals
\mathbb{E}\left[
f\left(\frac{1}{\sqrt{2}}G_1\right)
g\left(
\frac{t}{\sqrt{2}}G_1
+
\frac{\sqrt{1-t^2}}{\sqrt{2}}G_2
\right)
\right] \\
&=
\frac{1}{\pi^k(1-t^2)^{k/2}}
\int_{\mathbb{R}^k}\int_{\mathbb{R}^k}
f(x)g(y)
\exp\left(
\frac{-\|x\|_2^2-\|y\|_2^2+2t\langle x,y\rangle}{1-t^2}
\right)
\,dx\,dy .
\end{aligned}
\end{equation}

Then \(H_{f,g}\) extends to an analytic function on the strip
$
\left\{ z \in \mathbb{C} \colon \Re(z) \in (-1,1) \right\}.
$
We shall call \(\{f,g\}\) a \textbf{Krivine rounding scheme} if \(H_{f,g}\)
is invertible on a neighborhood of the origin, and if we consider the
Taylor expansion
\begin{equation}\label{one9}
H_{f,g}^{-1}(z)=\sum_{j=0}^{\infty} \widehat{a}_{2j+1}z^{2j+1},
\end{equation}
then there exists \(c=c(f,g)\in(0,\infty)\) satisfying
\begin{equation}\label{one10}
\sum_{j=0}^{\infty} |\widehat{a}_{2j+1}|c^{2j+1}=1.
\end{equation}
Only odd Taylor coefficients appear in \eqref{one9} since \(H_{f,g}\),
and therefore also \(H_{f,g}^{-1}\), is odd.
\end{definition}

A Krivine rounding scheme yields an algorithmic proof of Grothendieck's inequality in the following way.

\begin{algorithm}[\textbf{Krivine rounding algorithm}, \cite{braverman13}]\label{kalg}
\hfill

\underline{Input}: Vectors $\{x_r\}_{r=1}^{m},\{y_s\}_{s=1}^{n}\subseteq S^{m+n-1}$, and $f,g\colon\R^{k}\to\{-1,1\}$, which define coefficients $\{\widehat{a}_{2j+1}\}_{j=0}^{\infty}$ by \eqref{one9} and a constant $c\colonequals c(f,g)$ by \eqref{one10}. 

\underline{Output}:  Random signs $\{\sigma_r\}_{r=1}^{m},\{\tau_s\}_{s=1}^{n}\subseteq \{-1,1\}$.

\underline{Step 0} (Notation).
Consider the Hilbert space
\[
\mathcal{H}
=
\bigoplus_{j=0}^{\infty}
\left(\mathbb{R}^{m+n}\right)^{\otimes(2j+1)} .
\]
For any \(x\in S^{m+n-1}\) we can then define two vectors
\(I(x),J(x)\in\mathcal{H}\) by
\begin{equation}\label{one1213}
\begin{aligned}
I(x)
&\colonequals
\sum_{j=0}^{\infty}
|\widehat{a}_{2j+1}|^{1/2}
c^{(2j+1)/2}
x^{\otimes(2j+1)}\\
J(x)
&\colonequals\sum_{j=0}^{\infty}
\operatorname{sign}(\widehat{a}_{2j+1})
|\widehat{a}_{2j+1}|^{1/2}
c^{(2j+1)/2}
x^{\otimes(2j+1)},
\end{aligned}
\end{equation}

The choice of \(c\) was made in order to ensure
that \(I(x)\) and \(J(x)\) are unit vectors in \(\mathcal{H}\). Moreover,
the definitions \eqref{one1213} were made so that the following
identity holds:
\begin{equation}\label{one14}
\forall x,y\in S^{m+n-1},
\qquad
\langle I(x),J(y)\rangle_{\mathcal{H}}
\stackrel{\eqref{one9}}{=}
H_{f,g}^{-1}\bigl(c\langle x,y\rangle\bigr).
\end{equation}

\underline{Step 1} (Preprocessing the vectors).  Transform the
initial unit vectors
\[
\{x_r\}_{r=1}^{m},\{y_s\}_{s=1}^{n}\subseteq S^{m+n-1}
\]
to vectors
\[
\{u_r\}_{r=1}^{m},\{v_s\}_{s=1}^{n}\subseteq S^{m+n-1}
\]
satisfying the identities
\begin{equation}\label{one15}
\forall\,(r,s)\in\{1,\ldots,m\}\times\{1,\ldots,n\},
\qquad
\langle u_r,v_s\rangle
=
\langle I(x_r),J(y_s)\rangle_{\mathcal{H}}
\stackrel{\eqref{one14}}{=}
H_{f,g}^{-1}\bigl(c\langle x_r,y_s\rangle\bigr).
\end{equation}
As explained in \cite{alon04}, these new vectors can be computed efficiently
provided \(H_{f,g}^{-1}\) can be computed efficiently; this simply
amounts to computing a Cholesky decomposition.

\underline{Step 2} (Random projection).
Let \(G\colon\mathbb{R}^{m+n}\to\mathbb{R}^{k}\) be a random
\(k\times(m+n)\) matrix whose entries are i.i.d. standard Gaussian
random variables. Let
$
\sigma_1,\ldots,\sigma_m,\tau_1,\ldots,\tau_n\in\{-1,1\}
$
be random signs defined by
\begin{equation}\label{one16}
\forall\,(r,s)\in\{1,\ldots,m\}\times\{1,\ldots,n\},
\qquad
\sigma_r
\colonequals
f\left(\frac{1}{\sqrt{2}}Gu_r\right)
\quad\text{and}\quad
\tau_s
\colonequals
g\left(\frac{1}{\sqrt{2}}Gv_s\right).
\end{equation}
\end{algorithm}

A Krivine rounding scheme automatically upper bounds $K_G$ via the constant $c$ from \eqref{one10}.

\begin{cor}[{\cite{braverman13}}]\label{ckcor}
\emph{Let \(f,g\colon\mathbb{R}^{k}\to\{-1,1\}\) be a Krivine
rounding scheme. Then}
\[
K_G\le \frac{1}{c(f,g)}.
\]
\end{cor}
\begin{proof}[Proof, \cite{braverman13}]
Having obtained the random signs
\(\sigma_1,\ldots,\sigma_m,\tau_1,\ldots,\tau_n\in\{-1,1\}\) as in
\eqref{one16}, for every \(m\times n\) matrix \((a_{rs})\) we have
\[
\begin{aligned}
&\max_{\varepsilon_1,\ldots,\varepsilon_m,\delta_1,\ldots,\delta_n\in\{-1,1\}}
\sum_{r=1}^{m}\sum_{s=1}^{n}
a_{rs}\varepsilon_r\delta_s
\ge
\mathbb{E}\left[
\sum_{r=1}^{m}\sum_{s=1}^{n}
a_{rs}\sigma_r\tau_s
\right] \\
&\hspace{3cm}\stackrel{(*)}{=}
\mathbb{E}\left[
\sum_{r=1}^{m}\sum_{s=1}^{n}
a_{rs}H_{f,g}\bigl(\langle u_r,v_s\rangle\bigr)
\right] 
\stackrel{\eqref{one15}}{=}
c(f,g)
\sum_{r=1}^{m}\sum_{s=1}^{n}
a_{rs}\langle x_r,y_s\rangle,
\end{aligned}
\]
where \((*)\) follows by rotational invariance from \eqref{one16} and
\eqref{one8}.% We have thus proved the following corollary, which yields a
%systematic way to bound the Grothendieck constant from above.
\end{proof}
\begin{example}\label{sinex}
Let $k=1$ and let $f(x)=g(x)=\mathrm{sign}(x)$ for all $x\in\R$.  Then $H_{f,g}(z)=\frac{2}{\pi}\arcsin(z)$, $H_{f,g}^{-1}(z)=\sin(\pi z/2)$, $\widehat{a}_{2j+1}=(-1)^j(\pi/2)^{2j+1}/(2j+1)!$, so $c=c(f,g)$ satisfies
$$1=\sum_{j=0}^{\infty}|\widehat{a}_{2j+1}|c^{2j+1}
=\sum_{j=0}^{\infty}\frac{(\pi/2)^{2j+1}}{(2j+1)!}c^{2j+1}
=\sinh(\pi c/2).$$
So, $c=\frac{2}{\pi}\sinh^{-1}(1)=\frac{2}{\pi}\log(1+\sqrt{2})$, and Corollary \ref{ckcor} gives Krivine's bound $K_{G}\leq\frac{\pi}{2\log(1+\sqrt{2})}$.
\end{example}
% sinh(x) = (e^x + e^-x)/2 =1 when
%  x = log(1+sqrt(2)), since
%  e^x - e^-x = 1+sqrt(2) - 1/(1+sqrt(2))
%   = 1+sqrt{2} - (sqrt(2)-1) = 2
%

\subsection{The Proof of \texorpdfstring{\cite{braverman13}}{Braverman-Makarychev-Makarychev-Naor}}

The main result of \cite{braverman13} is \eqref{bmmneq}.  The proof of \eqref{bmmneq} in \cite{braverman13} did not directly use Corollary \ref{ckcor} or Algorithm \ref{kalg}.  Instead, they fix $k=2$ and some small $0<p<1$ and they modify \eqref{one16} so that, with probability $1-p$, \eqref{one16} is used to define $\sigma_1,\ldots,\sigma_m,\tau_1,\ldots,\tau_n$ with $f(x)=g(x)=\mathrm{sign}(x_2)$ $\forall$ $x\in\R^2$, and with probability $p$ \eqref{one16} is used with $f(x)=g(x)=\mathrm{sign}(x_2 - \eta h_5(x_1))$ $\forall$ $x\in\R^2$ where $\eta>0$ is small and $h_5$ denotes a fifth degree Hermite polynomial, defined in \eqref{eq:h5}. 
%(See the proof of Theorem 1.1 in \cite{braverman13} for more details.)  

In other words, the proof of \eqref{bmmneq} requires an extra randomization step.  This randomization and extra parameter $p$ are needed for technical reasons, to control the inverse of $H_{f,g}$.

However, this extra randomization step should be unnecessary to prove \eqref{bmmneq}, due to Corollary \ref{ckcor}.  Krivine rounding schemes themselves are in fact sufficient to prove Grothendieck's inequality \eqref{groeq} with its sharp constant (for $k\to\infty$ \cite{regev14}, with possibly additional randomness), but finding an explicit example of a Krivine rounding scheme that yields an explicit bound such as \eqref{bmmneq} remained an open problem.  Consequently, \cite{braverman13} conjectured the existence of the following specific Krivine rounding scheme.

\begin{conj}[{\cite[Conjecture 5.5]{braverman13}}]\label{conj0}
$\exists$ $\eta>0$ such that the functions $f_\eta(x)=g_\eta(x)=\mathrm{sign}(x_2 - \eta h_5(x_1))$ $\forall$ $x\in\R^2$ form a Krivine rounding scheme with $c(f_\eta,f_\eta)>\frac{2}{\pi}\log(1+\sqrt{2})$.  That is, \eqref{bmmneq} can be proven without the auxiliary randomization step and parameter $0<p<1$.
\end{conj}

Conjecture \ref{conj0} could be rephrased as asking: does there exist a Krivine rounding scheme $f,g\colon\R^{2}\to\{-1,1\}$ with $c(f,g)>\frac{2}{\pi}\log(1+\sqrt{2})$?  Note that \cite{braverman13} did not find such a Krivine rounding scheme, nor did \cite{regev14}.  (Note also that the Krivine rounding schemes of \cite{regev14} may require additional randomness beyond the definition of Krivine rounding scheme, as in the main result of \cite{braverman13}.)  Moreover, \cite{braverman13} showed in \cite[Lemma 2.4]{braverman13} that every Krivine rounding scheme $f,g\colon\R\to\{-1,1\}$ satisfies $c(f,g)\leq\frac{2}{\pi}\log(1+\sqrt{2})$ (with equality well known from Example \ref{sinex}).  So, a $k$-dimensional domain with $k\geq2$ is necessary in order to find $f,g\colon\R^{k}\to\{-1,1\}$ with $c(f,g)>\frac{2}{\pi}\log(1+\sqrt{2})$.

\subsection{Our Contribution}

Our first main result is to prove Conjecture \ref{conj0}.

\begin{theorem}[Main]\label{mainthm}
Conjecture \ref{conj0} holds.
\end{theorem}

Although an exact constant $c'>0$ is not specified for \eqref{bmmneq} in \cite{braverman13}, it seems that they showed $K_G<\frac{\pi}{2\log(1+\sqrt{2})}-10^{-500}$.  In removing their unnecessary extra parameter $p$, we then deduce an improved bound on $K_G$ using $f_\eta(x)=g_\eta(x)=\mathrm{sign}(x_2 - \eta h_5(x_1))$ $\forall$ $x\in\R^2$ in Definition \ref{kdef} with $\eta>0$ small.  That is, we show:

\begin{equation}\label{kg1}
K_G<\frac{\pi}{2\log(1+\sqrt{2})}-10^{-389}.
\end{equation}

As suggested in \cite{braverman13}, if we use instead $f_\eta(x)=\mathrm{sign}(x_2 - \eta h_3(x_1))$ and $g_\eta(x)=\mathrm{sign}(x_2 + \eta h_3(x_1))$ $\forall$ $x\in\R^2$, we obtain a better bound

\begin{equation}\label{kg2}
K_G<\frac{\pi}{2\log(1+\sqrt{2})}-10^{-217}.
\end{equation}
The bounds \eqref{kg2} and \eqref{kg1} are proven in Section \ref{secquant}.

Although the upper bound \eqref{kg2} is certainly far from the true constant $K_G$, it indicates that the true value of $K_G$ should be well-approximated by two-dimensional Krivine rounding schemes.  In other words, the main advance of Theorem \ref{mainthm} is conceptual, since it was previously unclear what explicit bounds Krivine rounding schemes could produce.

As in \cite{braverman13}, the proof of \eqref{kg1} and \eqref{kg2} is perturbative, relying on taking a small $\eta$ parameter such that various Taylor approximation errors are small.  It is unlikely such a perturbative approach could reveal the first few decimal places of $K_G$, due to the large approximation errors and small $\eta$ values.  It is therefore natural to try to directly (numerically) search for Krivine rounding schemes with reasonable constants, leading to improved $K_G$ bounds in Corollary \ref{ckcor}.  

This approach was already considered in \cite{braverman13}:
\begin{question}[{\cite[Question 3.2]{braverman13}}]\label{altq}
Are the optimizers of K\"{o}nig's bilinear functional \eqref{bkdef} an alternating Krivine rounding scheme (is $\mathrm{sign}(\widehat a_{2j+1})=(-1)^j$ for all $j\geq0$ in \eqref{one9})?
\end{question}
If Question \ref{altq} were true, then these optimizers would yield optimal estimates of $K_G$ by \cite[Equation (22)]{braverman13}.
However, our numerical simulations answer this question in the negative.  The purported optimizer of K\"{o}nig's bilinear functional, nicknamed the ``tiger partition'' (see Figures \ref{fig:tiger-f} and \ref{fig:tiger-g} below), appears to not be an alternating Krivine rounding scheme, as defined in \cite[Definition 2.2]{braverman13}.  
\begin{answer}\label{answer1}
We observe numerically that Question \ref{altq} seems to be false in dimension $k=2$, since the tiger partition's fifth coefficient $\widehat a_5$ appears to be negative.
\end{answer}
In other words, optimizers of K\"{o}nig's functional may be irrelevant for producing good rounding schemes (with improved bounds on $K_G$).  We discuss this numerical observation in Section \ref{sectiger}.  Since the tiger partition itself is not defined rigorously in \cite{braverman13} (in fact finding an analytic description of it is mentioned as a question in \cite[Question 3.1]{braverman13}), it is unclear how to formalize Answer \ref{answer1} using e.g. interval arithmetic, so we instead use floating-point numerical operations in Answer \ref{answer1}.  In summary, the tiger partition of \cite{braverman13} might be a red herring for improved $K_G$ bounds.  It could still occur that optimizers of K\"{o}nig's bilinear functional in $\R^k$ for $k>2$ are alternating Krivine rounding schemes, though we find no evidence of this from numerical calculations when $k=3,4$.

But if the tiger partition is not a good rounding scheme, what Krivine rounding scheme is good?  Using rigorous interval arithmetic, we prove that $f_\eta(x)=\mathrm{sign}(x_2 - \eta h_3(x_1))$ and $g_\eta(x)=\mathrm{sign}(x_2 + \eta h_3(x_1))$ does lead to reasonably good bounds on $K_G$ when $k=2$.  That is, we replace the perturbative approach of \eqref{kg2} by numerically computing a high degree Taylor expansion of $H_{f,g}$ in \eqref{htaylor} (and its inverse) together with a high degree tail error bound.  Since $f_\eta, g_\eta$ are defined with Hermite polynomials, the coefficients in the Taylor expansion have relatively simple expressions in terms of one-dimensional integrals that allow precise estimation.  Accurately computing these high degree Taylor expansions avoids the losses inherent from purely low degree expansions used to prove \eqref{kg1} and \eqref{kg2}.  We therefore find the following.
\begin{theorem}\label{thm2}
Using rigorous interval arithmetic, we show
$$
K_G
<\frac{\pi}{2\log(1+\sqrt{2})}-10^{-5}.
$$
\end{theorem}
We discuss the numerics used to verify Theorem \ref{thm2} further in Section \ref{seclast}.  Supporting numerical codes appear at the following URL:

\begin{urlpage}\label{oururl}
\url{https://github.com/sheilman77/grothendieck_upper_bounds}
\end{urlpage}

In fact, by perturbing the Hermite polynomial $h_3$ in the definition of $f,g$, we found numerical evidence that $K_G$ is at most
$$
\frac{\pi}{2\log(1+\sqrt{2})}-3\cdot 10^{-5}.
$$
However, to keep the length of this paper and the resulting codes more manageable, we leave the certification of this result to future work.

In light of the main result of \cite{regev14}, one might naturally ask if higher-dimensional ($k>2$) Krivine rounding schemes could be found numerically, that improve upon our observed $K_G$ bounds.  The larger $k$ is, the more difficult the numerical calculations would become, so again we leave this investigation for future work.  We simply note that our computations suggest that good Krivine rounding schemes could result from relatively simple choices of $f,g$, rather than a complicated choice of tiger partitions as put forward in \cite{braverman13}.

\subsection{Contrast with Previous Approaches}

How does our approach differ from that of \cite{braverman13}?  Our approach to Theorem \ref{mainthm} is broadly similar to the corresponding bound from \cite{braverman13}.  The main technical difference from \cite{braverman13} is in the treatment of the inverse of the $H$ function.  Instead of using an auxiliary parameter $p$ to control the inverse of $H$, we work directly with this inverse function itself. We prove a direct inverse expansion for $H$ in terms of $\eta$ in Section \ref{secinv} (after defining $q=\eta^2$), and then control the absolute inverse-coefficient sum with two different pieces, using an index cutoff $\log(1/\eta)$ in Section \ref{secnew}.  In contrast, \cite{braverman13} uses a fixed index cutoff.  In essence, the variable index cutoff supplants the auxiliary rounding parameter $p$ used in \cite{braverman13}.

However, we reiterate the conceptual difference between our result and \cite{braverman13}.  The main suggestion of \cite{braverman13} is that optimizers of K\"{o}nig's bilinear functional should lead to good bounds on $K_G$.  Such optimizers seemed so complicated that perhaps finding the exact value of $K_G$ could be a hopeless endeavor.  We instead advocate for a different path forward.  Theorem \ref{mainthm} and the resolution of Conjecture \ref{conj0} suggest instead that relatively simple choices of Krivine rounding schemes $f,g$ do lead to good bounds on $K_G$.  This realization then led to the successful rigorous computer-assisted proof of Theorem \ref{thm2} as detailed in Section \ref{seclast}.  So, perhaps $K_G$ can be found after all.

\section{Some Notation and Definitions}

Let $\Repart z$ denote the real part of $z\in\C$.  Let 
\begin{equation}\label{sldef}
        S\colonequals\{z\in\C\colon\abs{\Repart z}<1\},\qquad
        L\colonequals\log(1+\sqrt2).
\end{equation}
Thus \(L=\operatorname{arsinh}(1)\), so \(\sinh L=1\).  We use the Hermite
normalization of \cite{braverman13}: the polynomials \(h_m\) are orthonormal for the weight
\(e^{-x^2}\dd x\), i.e.
\begin{equation}\label{eq:hermite-normalization}
        \int_\R h_m(x)h_n(x)e^{-x^2}\dd x=\delta_{mn}.
\end{equation}
In particular,
\begin{equation}\label{eq:h5}
        h_5(x)=\frac{4x^5-20x^3+15x}{2\pi^{1/4}\sqrt{15}},
        \qquad h_3(x)=\frac{2x^{3}-3x}{\pi^{1/4}\sqrt{3}},
        \qquad\forall\,x\in\R.
\end{equation}
%H5 = 32 x^5 -..., divide sqrt(5! 2^5)
%                 =2^3 sqrt(5*4*3)
%H3 = 8x^3 - 12x, divide sqrt(2^3 3!)
%                 =4sqrt(3)
% h_3(x)=\frac{2x^{3}-3x}{\pi^{1/4}\sqrt{3}}
For \(\eta\in(0,1)\) and $d\in\{3,5\}$, define
\begin{equation}\label{eq:feta}
        f_\eta(x_1,x_2)=f_\eta^{(d)}(x_1,x_2)
        \begin{cases}
        1, & x_2\geq\eta h_d(x_1),\\
        -1, & x_2<\eta h_d(x_1).
        \end{cases}.
\end{equation}
Since \(h_d\) is odd for $d\in\{3,5\}$,
\(f_\eta\) is odd almost everywhere.  As in \cite[Equation (41)]{braverman13}, we use a different normalization for the function \eqref{one8} denoted $\forall$ $z\in S$ as
\begin{equation}\label{eq:paper-H}
        H_\eta(z)\colonequals\frac{1}{2\pi(1-z^2)}
        \int_{\R^2\times\R^2} f_\eta(x)f_\eta(y)
        \exp\!\left(\frac{-\norm{x}_2^2-\norm{y}_2^2+2z\ip{x}{y}}{1-z^2}\right)
        \dd x\dd y.
\end{equation}
When $d=3$, we use $f_{\eta}^{(3)}(x)(f_{-\eta}^{(3)}(y))$ in the integrand \eqref{eq:paper-H}.  By \cite[Lemma 4.3]{braverman13},
\begin{equation}\label{eq:H0-arcsin}
        H_0(z)=\arcsin z\qquad\forall\,z\in S.
\end{equation}

The function \eqref{one8} when applied to the pair
\((f_\eta,f_\eta)\) in dimension \(k=2\) with $d=5$ differs from \eqref{eq:paper-H} by a scalar:
\begin{equation}\label{eq:normalization-factor}
        H_{f_\eta,f_\eta}(z)=\frac{2}{\pi}H_\eta(z).
\end{equation}
Indeed, \eqref{one8} has prefactor \(\pi^{-2}(1-z^2)^{-1}\), whereas
\eqref{eq:paper-H} has prefactor \((2\pi)^{-1}(1-z^2)^{-1}\).  Let
\begin{equation}\label{eq:q-def}
        q\colonequals\eta^2,
        \qquad \widetilde H_q(z)\colonequals H_{\sqrt q}(z).
\end{equation}
When the local inverse of \(\widetilde H_q\) at the origin is defined, denote it by
\(B_q=\widetilde H_q^{-1}\), and write
\begin{equation}\label{eq:Bq-coefficients}
        B_q(z)=\sum_{m=0}^\infty a_{2m+1}(q)z^{2m+1},
\end{equation}
for all $z$ near $0$.  (For small $q$, this inverse will be shown to exist and be analytic in a neighborhood of the origin in Lemma \ref{lem:inverse-expansion} below.)  The even coefficients vanish because \(\widetilde H_q\) is odd, and the coefficients are real
since \(\widetilde H_q\) has real Taylor coefficients at the origin.

% \begin{theorem}\label{thm:main}
% Let $d=5$.  For all sufficiently small \(\eta>0\), the pair \((f_\eta,f_\eta)\) is a Krivine rounding
% scheme and
% \[
%         c(f_\eta,f_\eta)>\frac{2}{\pi}\log(1+\sqrt2).
% \]
% Consequently Conjecture 5.5 of \cite{braverman13} holds.
% \end{theorem}

% Similarly, when $d=3$, $c(f_\eta, f_{-\eta})>\frac{2}{\pi}\log(1+\sqrt2)$ for sufficiently small $\eta>0$.

\section{The first perturbation term of \texorpdfstring{$H$}{H}}

For any \(z\in S\) and $u,v\in\R$, define the two-dimensional Gaussian kernel
\begin{equation}\label{eq:pz}
        p_z(u,v)=\frac{1}{\pi\sqrt{1-z^2}}
        \exp\!\left(-\frac{u^2+v^2-2zuv}{1-z^2}\right).
\end{equation}
For real \(z=t\in(-1,1)\), this is the joint density of a centered Gaussian pair
\((U,V)\) with \(\E U^2=\E V^2=1/2\) and \(\E UV=t/2\).  To justify the estimates below, note that
\[
        \frac{u^2+v^2-2zuv}{1-z^2}
        =\frac{(u+v)^2}{2(1+z)}+\frac{(u-v)^2}{2(1-z)}.
\]
Since \(\Repart(1/(1+z))>0\) and \(\Repart(1/(1-z))>0\) for \(z\in S\), for every
compact set \(K\subset S\) there are constants \(c_K,C_K>0\) such that
\begin{equation}\label{eq:kernel-decay}
        \abs{\partial_u^\alpha\partial_v^\beta p_z(u,v)}
        \le C_{K,\alpha,\beta}(1+\abs{u}+\abs{v})^{\alpha+\beta}e^{-c_K(u^2+v^2)}
\end{equation}
for all \(z\in K\), $u,v\in\R$, $\alpha,\beta\geq0$.  In particular, the Gaussian integrals and the
threshold differentiations used below are locally uniform in \(z\in S\).

The following identities are well known, but we include their proofs for completeness.

\begin{lemma}[Hermite covariance]\label{lem:hermite-covariance}
For every \(m,n\ge0\) and every \(z\in S\),
\begin{equation}\label{eq:hermite-covariance}
        \iint_{\R^2} h_m(u)h_n(v)p_z(u,v)\dd u\dd v
        =\frac{z^n}{\sqrt\pi}\delta_{mn}.
\end{equation}
In particular, for $d\in\{3,5\}$
\begin{equation}\label{eq:h5-covariance}
        \iint_{\R^{2}} h_d(u)h_d(v)p_z(u,v)\dd u\dd v=\frac{z^d}{\sqrt\pi}.
%        \qquad
                % \iint h_3(u)h_3(v)p_z(u,v)\dd u\dd v=\frac{z^3}{\sqrt\pi}.
\end{equation}
Also,
\begin{equation}\label{eq:h5-moment}
        \iint_{\R^{2}} h_d(u)^2p_z(u,v)\dd u\dd v=\frac1{\sqrt\pi}.
\end{equation}
\end{lemma}

\begin{proof}
We first prove \eqref{eq:hermite-covariance}.  It suffices to prove the identity for real \(z=t\in(-1,1)\); the identity for all
\(z\in S\) follows by analyticity and the estimate \eqref{eq:kernel-decay}.  We use the
Hermite generating function in the normalization of \cite{braverman13}, namely
\begin{equation}\label{eq:generating-function}
        \sum_{n=0}^\infty h_n(x)\frac{s^n}{\sqrt{n!}}
        =\pi^{-1/4}\exp\!\left(\sqrt2sx-\frac{s^2}{2}\right).
\end{equation}
If \((U,V)\) has density \(p_t\), then \(\E U^2=\E V^2=1/2\) and \(\E UV=t/2\).  Hence 
\[
\begin{aligned}
&\E\left[\left(\sum_{m=0}^\infty h_m(U)\frac{s^m}{\sqrt{m!}}\right)
          \left(\sum_{n=0}^\infty h_n(V)\frac{r^n}{\sqrt{n!}}\right)\right]  \\
&\qquad =\pi^{-1/2}\exp\!\left(-\frac{s^2+r^2}{2}\right)
          \E e^{\sqrt2sU+\sqrt2rV}
       =\pi^{-1/2}e^{tsr}.
\end{aligned}
\]
The last equality follows from the known formula for the moment generating function of $(U,V)$.  Comparing the coefficient of \(s^mr^n\) gives \eqref{eq:hermite-covariance} for real
\(t\), and then analytic continuation gives the general case. 

Having proved \eqref{eq:hermite-covariance} (which then implies \eqref{eq:h5-covariance}), we now prove \eqref{eq:h5-moment}.  For real \(z=t\in(-1,1)\), completing the square gives
\[
\int_{\mathbb R}p_t(u,v)\,dv
\stackrel{\eqref{eq:pz}}{=}
\frac{e^{-u^2}}{\sqrt\pi}.
\]
Hence
\[
\iint_{\R^{2}} h_d(u)^2p_t(u,v)\,du\,dv
=
\frac1{\sqrt\pi}\int_{\mathbb R}h_d(u)^2e^{-u^2}\,du
\stackrel{\eqref{eq:hermite-normalization}}{=}
\frac1{\sqrt\pi}.
\]
The identity extends to all \(z\in S\) by analytic continuation.
\end{proof}

The following Lemma could be considered a routine extension of \cite[Theorem 4.4]{braverman13}, which used an expansion only for $z=i$.

\begin{lemma}[Expansion of \(\widetilde H_q\)]\label{lem:Hq-expansion}
Uniformly for $z$ on compact subsets of \(S\), as \(q\downarrow0\),
\begin{equation}\label{eq:Hq-expansion}
        \widetilde H_q(z)=\arcsin z+qQ(z)+O(q^2),
\end{equation}
where $\sigma\colonequals1$ for $d=5$, $\sigma\colonequals-1$ for $d=3$, and for $d\in\{3,5\}$, we have
\begin{equation}\label{eq:Q-def}
        Q(z)=Q_d(z)=\frac{2(\sigma z^d-z)}{\sqrt\pi\sqrt{1-z^2}}.
\end{equation}
The expansion \eqref{eq:Hq-expansion} may be differentiated once with respect to \(z\), locally uniformly on \(S\).
\end{lemma}
\begin{remark}\label{ubrk}
We make the constants from Lemma \ref{lem:Hq-expansion} explicit in Sections \ref{app1} and \ref{app2}.  When $d=5$, we have by \eqref{eq:H-remainder-product-final}, \eqref{eq:product-R108-final}, \eqref{eq:product-R108-final2} and \eqref{hqpbound} that
$$
\sup_{z\in \sin(1.08\cdot \D)}|\widetilde{H}_q(z) - (\arcsin z + q Q(z))|\leq q^{2}e^{23.6}.
$$
% using 1.09
$$
\sup_{z\in \sin(1.08\cdot \D)}\Big|\frac{d}{dz}\Big[\widetilde{H}_q(z) - (\arcsin z + q Q(z))\Big]\Big|\leq \frac{100}{\cos(1.08)}q^{2}e^{24.58}.
$$
When $d=3$, we have by \eqref{eq:H-remainder-product-final}, \eqref{eq:product-R108-final3}, \eqref{eq:product-R108-final4} and \eqref{hqpbound}
$$
\sup_{z\in \sin(1.08\cdot \D)}|\widetilde{H}_q(z) - (\arcsin z + q Q(z))|\leq q^{2}e^{14.5}.
$$
% using 1.09
$$
\sup_{z\in \sin(1.08\cdot \D)}\Big|\frac{d}{dz}\Big[\widetilde{H}_q(z) - (\arcsin z + q Q(z))\Big]\Big|\leq \frac{100}{\cos(1.08)}q^{2}e^{15}.
$$
\end{remark}

\begin{proof}
For any \(z\in S\), define
\begin{equation}\label{gzdef}
        G_z(a,b)\colonequals\iint_{\R^2}\sgn(u-a)\sgn(v-b)p_z(u,v)\dd u\dd v,
        \qquad \forall\,a,b\in\R.
\end{equation}
By \eqref{eq:kernel-decay}, \(G_z\) is \(C^4\) in the threshold variables \((a,b)\), with
bounds that are uniform for \(z\) in a fixed compact subset of \(S\).  Direct differentiation then gives %\snote{need to double check these derivative calculations}
\begin{equation}\label{eq:first-derivatives-G}
        \partial_aG_z(a,b)=-2\int_\R \sgn(v-b)p_z(a,v)\dd v,
        \qquad
        \partial_bG_z(a,b)=-2\int_\R \sgn(u-a)p_z(u,b)\dd u.
\end{equation}
At \((a,b)=(0,0)\) these derivatives vanish, because the integrands are odd.  A second
threshold differentiation gives
\begin{equation}\label{eq:G-second-derivatives}
        \partial_{aa}G_z(0,0)=\partial_{bb}G_z(0,0)
        =-\frac{4z}{\pi\sqrt{1-z^2}},
        \qquad
        \partial_{ab}G_z(0,0)=\frac{4}{\pi\sqrt{1-z^2}}.
\end{equation}
For instance,
\[
\begin{aligned}
        \partial_{aa}G_z(0,0)
        &=-2\int_\R \sgn(v)\partial_up_z(0,v)\dd v                                      \\
        &=-\frac{4z}{1-z^2}\int_\R |v|\frac{\exp(-v^2/(1-z^2))}{\pi\sqrt{1-z^2}}\dd v
        =-\frac{4z}{\pi\sqrt{1-z^2}},
\end{aligned}
\]
where \(\Repart(1/(1-z^2))>0\) on \(S\).  Also
\(\partial_{ab}G_z(a,b)=4p_z(a,b)\), which yields the displayed formula for
\(\partial_{ab}G_z(0,0)\).

Let \((X_1,Y_1)\) be distributed according to the kernel \(p_z\).  From
\eqref{eq:paper-H} and \eqref{eq:pz}, or first for real \(z\) and then by analyticity, we have
\begin{equation}\label{eq:Hq-conditional}
        \widetilde H_q(z)
        \stackrel{\eqref{eq:q-def}}{=}H_{\sqrt{q}}(z)
        \stackrel{\eqref{eq:paper-H}}{=}\frac\pi2\,
        \E\left[G_z\big(\sqrt q\,h_d(X_1),\sqrt q\,\sigma h_d(Y_1)\big)\right],
\end{equation}
where $\sigma=1$ for $d=5$ and $\sigma=-1$ for $d=3$.  (Since \eqref{eq:paper-H} has an absolutely convergent integrand, $H_{\sqrt{q}}(z)$ is analytic in the strip $S$.)  Write \(A=h_d(X_1)\) and \(B=\sigma h_d(Y_1)\).  Taylor expansion of \(G_z\) at \((0,0)\) to
third order, with integral remainder, gives
\[
\begin{aligned}
G_z(\sqrt q A,\sqrt q B)
&=G_z(0,0) \\
&\quad +\frac q2\Big(\partial_{aa}G_z(0,0)A^2
      +2\partial_{ab}G_z(0,0)AB
      +\partial_{bb}G_z(0,0)B^2\Big) \\
&\quad +\frac{q^{3/2}}{6}\Big(
      \partial_{aaa}G_z(0,0)A^3
      +3\partial_{aab}G_z(0,0)A^2B \\
&\qquad\qquad\qquad
      +3\partial_{abb}G_z(0,0)AB^2
      +\partial_{bbb}G_z(0,0)B^3\Big) \\
&\quad +O_K\big(q^2(\abs A+\abs B)^4\big),
\end{aligned}
\]
uniformly for $z$ in a compact subset of $S$.  The cubic expectation is zero:
the kernel \(p_z(x,y)\dd x\dd y\) is invariant under \((x,y)\mapsto(-x,-y)\), and each
of \(A^3,A^2B,AB^2,B^3\) is odd under this simultaneous sign change.  Since \(h_d\) is a
polynomial, the fourth moment in the remainder is finite and locally uniformly bounded in
\(z\).  Hence by \eqref{eq:Hq-conditional}
\begin{equation}\label{eq:Hq-second-order}
\begin{aligned}
\widetilde H_q(z)
&=\widetilde H_0(z)
  +\frac\pi2q\left(
      \partial_{aa}G_z(0,0)\E A^2
      +\partial_{ab}G_z(0,0)\E AB\right)+O_K(q^2).
\end{aligned}
\end{equation}
Here we used \(\partial_{aa}G_z=\partial_{bb}G_z\) by \eqref{eq:G-second-derivatives} and \(\E A^2=\E B^2\).
By Lemma~\ref{lem:hermite-covariance}, \(\E A^2=1/\sqrt\pi\) and
\(\E AB=\sigma z^d/\sqrt\pi\).  Substituting \eqref{eq:G-second-derivatives} into
\eqref{eq:Hq-second-order} yields
\[
        \widetilde H_q(z)=\widetilde H_0(z)+q\frac{2(\sigma z^d-z)}{\sqrt\pi\sqrt{1-z^2}}+O_K(q^2).
\]
Together with \eqref{eq:H0-arcsin}, this is \eqref{eq:Hq-expansion}.  The locally uniform
expansion of the derivative follows from Cauchy's integral formula, applied on a slightly
larger compact subset of \(S\).
\end{proof}

\section{The inverse function and its coefficients}\label{secinv}

In this Section we modify and extend \cite[Lemmas 5.3 and 5.4]{braverman13}.

\begin{lemma}[Inverse perturbation]\label{lem:inverse-expansion}
Fix \(0<r\leq1.08\).  For all sufficiently small \(q>0\), the inverse branch
\(B_q=\widetilde H_q^{-1}\) is analytic on \(r\D\).  Uniformly for \(z\in r\D\),
\begin{equation}\label{eq:Bq-expansion}
        B_q(z)=\sin z+q\Psi(z)+O_r(q^2),
\end{equation}
where for $d=5$, we have
\begin{equation}\label{eq:Psi-def}
        \Psi(z)=-Q(\sin z)\cos z
        \stackrel{\eqref{eq:Q-def}}{=}\frac{2}{\sqrt\pi}\big(\sin z-\sin^5z\big)
        =\frac{6\sin z+5\sin(3z)-\sin(5z)}{8\sqrt\pi}.
\end{equation}
Consequently, if
\[
        \sin z=\sum_{\substack{n\ge1\\ n\text{ odd}}}b_nz^n,
        \qquad
        \Psi(z)=\sum_{\substack{n\ge1\\ n\text{ odd}}}c_nz^n,
\]
then, for \(n=2m+1\),
\begin{equation}\label{eq:bn-cn}
        b_n=\frac{(-1)^m}{n!},
        \qquad
        c_n=\frac{(-1)^m}{8\sqrt\pi\,n!}\big(6+5\cdot3^n-5^n\big).
\end{equation}
In particular, there is an absolute constant \(C_0\) such that
\begin{equation}\label{eq:cn-bound}
        \abs{c_n}\le C_0\frac{5^n}{n!}
        \qquad\forall\text{ odd }n\geq1.
\end{equation}
Moreover, for every $0<\rho\leq1.08$ there are constants \(q_\rho>0\) and \(C_\rho<\infty\) such
that, for all \(0<q<q_\rho\) and all odd \(n\ge1\), the Taylor coefficients $a_n(q)$ of $B_q$ \eqref{eq:Bq-coefficients} satisfy
\begin{equation}\label{eq:coefficient-error}
        \abs{a_n(q)-b_n-qc_n}\le C_\rho q^2\rho^{-n}.
\end{equation}
\end{lemma}
\begin{remark}\label{d3invrk}
In the $d=3$ case, we have instead
% sin^3 = (1/4)(3sin - sin3)
\begin{equation}
\Psi(z)=\frac{2}{\sqrt{\pi}}(\sin(z)+\sin^3(z))
=\frac{7\sin z - \sin(3z)}{2\sqrt{\pi}},\qquad
b_{n}=\frac{(-1)^{m}}{n!},\qquad
c_{n}=\frac{(-1)^{m}}{2\sqrt{\pi}n!}(7 - 3^n),
\end{equation}
\begin{equation}
|c_n|\leq C_0 \frac{3^n}{n!},\qquad
|a_n(q)-b_n -qc_n|\leq C_\rho q^2 \rho^{-n}
\end{equation}
\end{remark}
%\snote{editing to enlarge $r$}
\begin{proof}
Let \(K_r=\sin(\overline{r\D})\).  Since \(r<\pi/2\),
\(\arcsin(\sin z)=z\) for \(\abs z\le r\), and \(\cos z\) is bounded away from zero on
\(\overline{r\D}\).  

We need to check that \(K_R=\{\sin\zeta\colon\ |\zeta|\le R\}\) is contained
compactly in the strip \(S=\{w\in\C\colon\ |\Re w|<1\}\).  We shall use
\(R=27/25\).  Write \(\zeta=x+iy\).  Then
\[
|\Re(\sin\zeta)|
=|\sin x\cosh y|
=|\sin x|\cosh y.
\]
% sin(x+iy)= sin x cos(iy)+cos x sin(iy)
% = sin x cosh y + cos x i sinh y
% Re sinz = sinx cosh y
% Im sinz = cosx sinh y
% %
% \[
% |\Re(\sin\zeta)|=|\sin x|\cosh y.
% \]
We use the elementary inequalities
\[
\cosh y\le e^{y^2/2},\quad\forall\,y\in\R
,\qquad
\sin x\le x e^{-x^2/6},\quad\forall\,0\leq x<\pi\
\]
% The first follows from \(\log\cosh y\le y^2/2\).  For the second, set
% \[
% F(x)=\log\frac{\sin x}{x}+\frac{x^2}{6}.
% \]
% Using
% \[
% \cot x=\frac1x-2x\sum_{n=1}^{\infty}\frac1{n^2\pi^2-x^2},
% \qquad 0<x<\pi,
% \]
% we get
% \[
% F'(x)=\cot x-\frac1x+\frac{x}{3}
% \le
% -2x\sum_{n=1}^{\infty}\frac1{n^2\pi^2}+\frac{x}{3}
% =
% -2x\cdot\frac16+\frac{x}{3}=0.
% \]
% Since \(F(0)=0\), this proves \(\sin x\le xe^{-x^2/6}\).

Now if \(|\zeta|\le R\), then \(x^2+y^2\le R^2\), so
\[
|\Re(\sin\zeta)|
\le
|x| e^{-x^2/6}e^{y^2/2}
\le
|x|\exp\left(\frac{R^2}{2}-\frac{2x^2}{3}\right).
\]
The right-hand side is maximized at \(x=\sqrt3/2\), so plugging this in and using $R=27/25$,
\[
|\Re(\sin\zeta)|
\le
\frac{\sqrt3}{2}\exp\left(\frac{R^2-1}{2}\right)
\leq \frac{\sqrt{3}}{2}e^{\frac{104}{2\cdot 625}}
<\frac{19}{20}.
\]
% e^(x^2 -1)/2 = 2/sqrt(3)
% x^2 -1 = 2 log(2/sqrt(3))
%x^2 = 1+ 2 log(2/sqrt(3))
% s = sqrt(1+2*log(2/sqrt(3)))
%
% With \(R=27/25\), this gives
% \[
% |\Re(\sin\zeta)|^2
% \le
% \frac34 e^{104/625}.
% \]
% Since \(e^t<1/(1-t)\) for \(0<t<1\),
% \[
% e^{104/625}<\frac{625}{521}.
% \]
% Therefore
% \[
% |\Re(\sin\zeta)|^2
% <
% \frac34\cdot\frac{625}{521}
% =
% \frac{1875}{2084}
% <
% \left(\frac{19}{20}\right)^2.
% \]
% Thus
% \[
% |\Re(\sin\zeta)|<\frac{19}{20}
% \qquad(|\zeta|\le 27/25).
% \]
In particular \(K_{27/25}\Subset S\).
%The branch
%(H_0=\arcsin\) is univalent on \(S\).  
Choose \(s\) with \(r<s<\pi/2\), and choose an open set
\(U\) such that
\[
K_r\Subset K_s\Subset U\Subset S.
\]
Then choose \(0<\rho_0\le \rho_{r,s}\), so that the disk $D(w_0,\rho_0)$ satisfies \(D(w_0,\rho_0)\subset K_s\) for every \(w_0\in K_r\).  On \(U\), the derivative
\[
        H_0'(w)=(\arcsin)'(w)=\frac1{\sqrt{1-w^2}}
\]
is bounded away from zero since $U\subset S$.  Then Lemma \ref{lemma5} below implies that, for any $w$ with $|w-w_0|\leq\rho_{0}$, we have
%, and that there is a constant
%\(m_s\colonequals1/\cosh(s)>0\) such that
\begin{equation}\label{eq:local-lower-lipschitz}
        \abs{H_0(w)-H_0(w_0)}\ge \frac{1}{\cosh s}\abs{w-w_0}.
\end{equation}

Lemma~\ref{lem:Hq-expansion} gives some $C_U>0$ such that, for all sufficiently small \(q\),
\begin{equation}\label{eq:Hq-H0-uniform}
        \sup_{w\in U}|\widetilde H_q(w)-H_0(w)|\le C_Uq.
\end{equation}
On the circle \(\{w\in\C\colon\abs{w-\sin z}=\rho_0\}\subset U\), the lower bound
\eqref{eq:local-lower-lipschitz} gives
\[
        \abs{H_0(w)-z}=\abs{H_0(w)-H_0(\sin z)}\ge \frac{1}{\cosh s}\rho_0.
\]
So, for \(q\) small enough such that \(C_Uq<\frac{1}{\cosh s}\rho_0\), Rouch\'{e}'s theorem implies that the function
\[w\mapsto\widetilde H_q(w)-z=\widetilde H_q(w)-H_0(w) +H_0(w) - z
\]
has exactly one zero in \(D(\sin z,\rho_0)\).  Denote it by
\(w_q=B_q(z)\).  Since
\(\widetilde H_q'(w_q)=H_0'(w_q)+O(q)\) by Lemma \ref{lem:Hq-expansion} and \(H_0'\) is bounded away from zero on \(U\), the
implicit function theorem shows that \(B_q\) depends analytically on \(z\).  This is the
inverse branch of \(\widetilde H_q\) on \(r\D\).

% Since the differentiated expansion in Lemma~\ref{lem:Hq-expansion} gives
% \[
%         \sup_{w\in U}|\widetilde H_q'(w)-H_0'(w)|=O(q),
% \]
% and since \(H_0'\) is bounded away from zero on \(U\), we may decrease \(q\) so that
% \(\widetilde H_q'\) has no zeros on \(U\). Applying the holomorphic implicit
% function theorem to
% \[
%         F_q(w,z)=\widetilde H_q(w)-z
% \]
% at the Rouché zero \(w=w_q(z)\) gives a local holomorphic branch of solutions.
% The uniqueness from Rouché then glues these local branches together, so
% \(B_q(z)=w_q(z)\) is holomorphic on \(r\mathbb D\).

We now prove the sharper estimate \(B_q(z)=\sin z+O_r(q)\).  Put
\begin{equation}\label{weq}
        w_0=\sin z,
        \qquad w_q=B_q(z).
\end{equation}
Since \(\widetilde H_q(w_q)=z=H_0(w_0)\), \eqref{eq:Hq-H0-uniform} gives
\[
        |H_0(w_q)-H_0(w_0)|
        =|H_0(w_q)-\widetilde H_q(w_q)|
        \le C_Uq.
\]
Using \eqref{eq:local-lower-lipschitz} with $w=w_q$,
\begin{equation}\label{eq:wq-w0-Oq}
        \abs{w_q-w_0}\le \cosh(s)C_{U}q.
\end{equation}
Therefore, \(B_q(z)=\sin z+O_r(q)\) by \eqref{weq}, uniformly for
\(z\in r\D\).

Next write the expansion from Lemma~\ref{lem:Hq-expansion} on \(U\) in the form
\begin{equation}\label{eq:Hq-with-remainder}
        \widetilde H_q(w)=H_0(w)+qQ(w)+q^2R_q(w),
        \qquad \sup_{w\in U}\abs{R_q(w)}\le \widetilde{C}_U.
\end{equation}
Since \(H_0(w_q)=H_0(w_0)+H_0'(w_0)(w_q-w_0)+O_r(\abs{w_q-w_0}^2)\), and since
\eqref{eq:wq-w0-Oq} gives \(\abs{w_q-w_0}=O_r(q)\), we have
\begin{equation}\label{eq:H0-Taylor-detail}
        H_0(w_q)-H_0(w_0)=H_0'(w_0)(w_q-w_0)+O_r(q^2).
\end{equation}
Similarly, \(Q\) is analytic on \(U\), hence
\begin{equation}\label{eq:Q-Taylor-detail}
        qQ(w_q)=qQ(w_0)+qO_r(\abs{w_q-w_0})=qQ(w_0)+O_r(q^2).
\end{equation}
Finally, \(q^2R_q(w_q)=O_r(q^2)\).  Substituting \eqref{eq:Hq-with-remainder} (with $w=w_q$), \eqref{eq:H0-Taylor-detail} and \eqref{eq:Q-Taylor-detail} into
\[
        0=\widetilde H_q(w_q)-H_0(w_0)
\]
which follows from \(\widetilde H_q(w_q)=z=H_0(w_0)\), gives
\begin{equation}\label{eq:inverse-linear-equation}
        0=H_0'(w_0)(w_q-w_0)+qQ(w_0)+O_r(q^2).
\end{equation}
That is,
\[
        0=(\arcsin)'(w_0)(w_q-w_0)+qQ(w_0)+O_r(q^2).
\]
On \(r\D\), \((\arcsin)'(\sin z)=1/\cos z\) and \(\sqrt{1-\sin^2z}=\cos z\).  Hence
\[
        w_q-w_0=-qQ(\sin z)\cos z+O_r(q^2)
        \stackrel{\eqref{eq:Q-def}}{=}\frac{2q}{\sqrt\pi}(\sin z-\sin^5z)+O_r(q^2),
\]
which proves \eqref{eq:Bq-expansion}.  The trigonometric identity in \eqref{eq:Psi-def}
follows from
\[
        \sin^5z=\frac{10\sin z-5\sin(3z)+\sin(5z)}{16}.
\]
Expanding the sine functions gives \eqref{eq:bn-cn}, and \eqref{eq:cn-bound} is immediate.
Finally, applying Cauchy's integral formula to the remainder in \eqref{eq:Bq-expansion}
on the circle \(\abs z=\rho\) gives \eqref{eq:coefficient-error}.
\end{proof}

\begin{remark}\label{exrk}
In the above proof we may take $C_{0}=1/4$ in \eqref{eq:cn-bound} and we take $C_{\rho}$ to be the same constant from \eqref{eq:coefficient-error} as in \eqref{eq:Bq-expansion}.  Then Lemma \ref{lemma5} gives $|H_{0}(w)-H_{0}(w_{0})|\geq\frac{1}{\cosh(s)}|w-w_0|$.

Also \eqref{eq:Hq-H0-uniform} can be written as, for all $w\in U$,
$$|\widetilde H_q(w)-H_0(w)|\leq q|Q(w)|+O(q^{2})
\leq 5q+q^{2}e^{24.6},$$
by numerics in Section \ref{app1}, i.e. \eqref{eq:H-remainder-product-final} and
\eqref{eq:product-R108-final2}, so we can use (for $d=5$) $$C_{U}\colonequals 5+qe^{24.6}.$$
Then after \eqref{eq:Hq-with-remainder} the $O_r(\abs{w_q-w_0}^2)$ is bounded by $9|w_q -w_0|^2$ by upper bounding the second derivative of $\arcsin$ by \eqref{fif3}.  Then the implied constant $O_{r}(q^{2})$ in \eqref{eq:H0-Taylor-detail} is at most $9(C_{U}\cosh(s))^{2}q^{2}$.  Then denoting $C_Q$ as the first implied constant in \eqref{eq:Q-Taylor-detail}, the $O_{r}(q^{2})$ term in \eqref{eq:Q-Taylor-detail} is at most $C_Q (C_U \cosh(s))q^{2}$, with $C_Q\leq8.5$ 
by \eqref{fif5} below.

And then $q^{2}R_{q}(w_{q})$ is bounded by $\widetilde{C}_U q^{2}$ by \eqref{eq:Hq-with-remainder}, and we may take $\widetilde{C}_U =e^{24.6}$.  Next, the $O_{r}(q^{2})$ term in \eqref{eq:inverse-linear-equation} is at most $q^{2}$ times
$$
\widetilde{C}_U+9(C_U \cosh(s))^2 +C_Q(C_U \cosh(s))
$$
the same implied constant bounds then carry through to the next displayed equations, after also multiplying by $1.643$ to account for the $O_r(q^2)$ term being multiplied by $\cos(z)$ on the set $|z|\leq 1.08$.

% \snote{does multiplying by cos z change anything? double check.  $|\cos z|^2 = \cos^2 x \cosh^2 y+\sin^2 x \sinh^2 y$.  $\cosh^2 = 1+\sinh^2$, so $|\cos z|^2 = \cos^2 x (1+\sinh^2 y)+\sin^2 x \sinh^2 y=\cos^2 x+\sinh^2 y$.  If $z=x+iy\in r D$, then max occurs at $x=0,y=1.08$, which is $1+\sinh^2(1.08)=\cosh^2(1.08)$, after taking square root, get at most $1.643$.}

When $d=5$, we may choose $r=1.08$, $s=1.09$, $\cosh(s)=\cosh(1.09)<5/3$
%$m_{r,s}=(.01/2)\cos(1.085)\sin(1.085)/1.085>1/530$
, $\widetilde{C}_U=e^{24.6}$, $C_{0}=1/4$, so the $O_{r}(q^{2})$ term in \eqref{eq:inverse-linear-equation} can be bounded by $q^{2}$ times
$$e^{24.6}+9([5+qe^{24.6}]\cdot (5/3))^2 + 8.5([5+qe^{24.6}]\cdot (5/3))
\leq q^{2}e^{63}+qe^{31}+e^{25}.$$
% e^24.6+
% e^49+12.55 -1.38  * q^2
% e^1.16+3.22+12.55
% e^.7+2.303+12.55 * q
% +e^2.14+24.4+12.55
% +e^2.14+1.61+.52
% +e^2.14+24.6+.52 * q
% = e^24.6 + e^62.71 q^2
%    +e^16.93 + q*e^15.6  
%    +e^4.27 + e^27.26 *q
% <= e^63 q^2 + q e^28 + e^25
%
Likewise, we can use $C_{\rho}=q^{2}e^{63}+qe^{31}+e^{25}$.
%\leq q^{2}e^{63}+qe^{31}+e^{25}
%$ when $\rho>1$.
When $d=3$, we can choose $\widetilde{C}_U\colonequals e^{15}$, $C_{0}=1/2$, $C_{Q}=23$ by \eqref{eq:product-R108-final4} and \eqref{fif4}, so the $O_{r}(q^{2})$ term in \eqref{eq:inverse-linear-equation} can be bounded by $q^{2}$ times
$$e^{15}+9([3+qe^{15}]\cdot (5/3))^2 + 23([3+qe^{15}]\cdot (5/3))
\leq q^{2}e^{42}+qe^{29}+e^{15.5}.$$
% e^15 +
% e^14.1 + 
% q e^ 13.64+15
% q^2 e^30+12
% e^11.01
% q e^15 + 9.5
%
\end{remark}

\section{The absolute coefficient sum at \texorpdfstring{$L$}{L}}\label{secnew}

Let $a_n(q)$ be the Taylor coefficients of $B_q$, as in \eqref{eq:Bq-coefficients}.  For \(t\ge0\), define
\begin{equation}\label{eq:Aq-def}
        A_q(t)=\sum_{\substack{n\ge1\\ n\text{ odd}}}\abs{a_n(q)}t^n,
\end{equation}
whenever this series converges.  Recalling \eqref{one10} and Corollary \ref{ckcor}, an upper bound on $K_G$ results from understanding where $A_q$ takes the value $1$.  Since we anticipate $A_q(t)$ taking the value $1$ at some $t>L$, we will eventually show that $A_q(L)$ is less than $1$ for small $q$ in Section \ref{secquant}.  To prepare for that result, we compare the absolute coefficient sum $A_q(L)$ with the same alternating sum without absolute values, i.e. $B_q(iL)/i$.  A similar strategy with different details was used in \cite[Theorem 5.1]{braverman13}, which required an auxiliary parameter $p$.

\begin{lemma}\label{lem:absolute-vs-alternating}
As \(q\downarrow0\),
\begin{equation}\label{eq:absolute-vs-alternating}
        A_q(L)=\frac{B_q(iL)}{i}+o(q^2).
\end{equation}
The $o(q^{2})$ term is bounded by $6 C_\rho q^{2} q^{\frac{0.3537}{\log\log(1/q)}}+q^{2.01}$, if $q<e^{-390}$ and $d=5$.  
If $d=3$, the $o(q^2)$ term is at most $6 C_\rho q^{2} q^{\frac{0.3537}{\log\log(1/q)}}+q^{2.01}$, if $q<\min(e^{-81},(4C_\rho)^{-5})$.
\end{lemma}
\begin{proof}
We first prove the case $d=5$, saving $d=3$ for Remark \ref{d3l5rk}.  Fix $\rho\colonequals1.08$.  By \eqref{sldef}
\begin{equation}\label{eq:r-choice}
        L<.89<1<\rho.
\end{equation}
All coefficient estimates below use Lemma~\ref{lem:inverse-expansion} with this value of
\(\rho\).  For odd \(n\), put
\begin{equation}\label{sdedef}
        s_n\colonequals(-1)^{(n-1)/2},\qquad d_n(q)\colonequals b_n+qc_n,
        \qquad e_n(q)\colonequals a_n(q)-d_n(q).
\end{equation}
Thus \(b_n=s_n/n!\).  By \eqref{eq:bn-cn},
\begin{equation}\label{eq:delta-n-def}
        s_nd_n(q)=\frac{1+q\delta_n}{n!},
        \qquad
        \delta_n\colonequals\frac{6+5\cdot3^n-5^n}{8\sqrt\pi},
\end{equation}
and \(\abs{\delta_n}\le C5^n\) with $C=1/4$.  Also, by \eqref{eq:coefficient-error},
\begin{equation}\label{eq:en-bound}
        \abs{e_n(q)}\le C_\rho q^2\rho^{-n}.
\end{equation}

Let \(T\colonequals\log(1/q)\), and define
\begin{equation}\label{eq:N0-N1}
        N_1\colonequals\left\lfloor1.75\frac{T}{\log T}\right\rfloor,
        \qquad
        N_0\colonequals\left\lfloor3.7\frac{T}{\log T}\right\rfloor .
\end{equation}
%for x = 100:110, x,-x + (2.2*log(5))*(x/log(x))  , end
We first prove that \(a_n(q)\) has the alternating sign \(s_n\) for every odd
\(n\le N_1\), when \(q\) is small enough.  Indeed, by definition of $T$ and $N_{1}$
\[
        \max_{n\le N_1} q\abs{\delta_n}
        \le C q5^{N_1}
        \leq C\exp\big(-T+(3.5\log 5)T/(2\log T)\big)=o(1),
\]
so \(s_nd_n(q)\ge \frac{1}{2}n!^{-1}\) for all odd \(n\le N_1\) and small \(q\) (since $-T+(3.5\log 5)T/(2\log T)<-(1/3)\cdot T$ when $T>69$, it suffices to choose $q<\min(e^{-70},2^{-6}C^{-3})$.)

%3.5ln5 ~ 5.633
% need |q\delta_n|<1/2, i.e. need
% C_r exp(-T + 3ln5 T/2log T)<1/2
% -t+3.5ln5t/2logt < log(1/2cr)
% choose q small such that this happens
% T = log(1/q)
% if T>25,  then
% -t+3.5ln5t/2log t< -t/3
% so need -t/3 < log(1/2cr)
% -log(1/q)<3 log(1/2cr)
% q< (1/ 4cr)^(3)
% q< 2^{-6}c_r^-(5/2)
% and log log(1/q)>2
%  log(1/q)>e^2
% 1/q > e^e^2
% q< e^- e^2
%
%  need -t + (3*log 5)t / (2*log t)
%      < - t/4
%  happens when t>= 25

On the
other hand, Stirling's formula and \eqref{eq:en-bound} give
\[
\begin{aligned}
        \max_{n\le N_1} n!\abs{e_n(q)}
        &\le C_\rho q^2\max_{n\le N_1}n!\rho^{-n},\\
        \log\big(q^2N_1!\rho^{-N_1}\big)
        &\leq-2T+N_1\log N_1-N_1 \log \rho=-\frac14T+o(T),
\end{aligned}
\]
% log n!<= (n+1/2)log n - n + .5log 2pi
%   + 1/12n
%  if n>3, get
%  log n! <= n log n
%
% -2T + (3.5/2)(T/ log T)log((3/2)T/Log T)
%  <= .4*T  for T>=4
%
%for x = 4:30, x,-2*x + (3.5/2)*(x/log(x)) * log((3.5/2)*x / log(x)) -(3.5/2)*(x/log(x))*log(1.08) + .4*x, end
%
%for x = 200:300, (-.25+.4)*x  + 1.75*(x/log(x))*(log(1.75) - log(log(x)))-(3.5/2)*(x/log(x))*log(1.08) , end
%
%  x>70
which tends to \(-\infty\).  (It is bounded by $-.2 T$ when $T\geq70$).  Hence \(\abs{e_n(q)}\le \frac14n!^{-1}\) for all odd
\(n\le N_1\), and therefore \(s_na_n(q)>0\) for every such \(n\), by \eqref{sdedef} using also \(s_nd_n(q)\ge \frac{1}{2}n!^{-1}\).
%
% need  C_r q^2... < 1/4
%  log q^2 ... < -.4T when T>=4
%  so need log q^2... < log 1/4 - log Cr
%  log q^2 ... < log 1/4Cr
%  know log q^2 ... < -.4T when T>=4
%  so need T>=4  and
%   -.4T < log 1/4Cr
%  -.4 log 1/q < log 1/4Cr
%  log 1/q > -1/.4 log 1/4Cr
%  1/q > (1/4Cr)^-1/.4=(4Cr)^1/4
%  q< (4Cr)^-1/.4
(It suffices to choose $q< \min(e^{-6}, (4C_\rho)^{-5})$.)  Since \(B_q(iL)/i=\sum_{\substack{n\ge1\\ n\text{ odd}}}s_n a_n(q)L^n\) by \eqref{eq:Bq-coefficients},
\begin{equation}\label{eq:difference-nonnegative}
\begin{aligned}
        0\le A_q(L)-\frac{B_q(iL)}{i}
        &=\sum_{\substack{n\ge1\\ n\text{ odd}}}
          \big(\abs{a_n(q)}-s_na_n(q)\big)L^n .
\end{aligned}
\end{equation}
Since $s_n a_n(q)>0$ for all $n\leq N_1$, those terms vanish in the right sum of \eqref{eq:difference-nonnegative}.  For \(n>N_1\), use the elementary inequality
\begin{equation}\label{eq:elementary-ineq}
        \abs{d+e}-s(d+e)\le \abs d-sd+2\abs e,
        \qquad s\in\{-1,1\},\quad d,e\in\R.
\end{equation}
It then follows from \eqref{eq:difference-nonnegative} that
\begin{equation}\label{eq:split-tail}
\begin{aligned}
        A_q(L)-\frac{B_q(iL)}{i}
        &\le \sum_{\substack{n>N_1\\ n\text{ odd}}}
            \big(\abs{d_n(q)}-s_nd_n(q)\big)L^n
          +2\sum_{\substack{n>N_1\\ n\text{ odd}}}\abs{e_n(q)}L^n .
\end{aligned}
\end{equation}

The error tail is \(o(q^2)\), since \(L=\log(1+\sqrt{2})<\rho=1.08\), $L/\rho<.817$ and \eqref{eq:en-bound} imply
\begin{equation}\label{eq:error-tail-o}
\begin{aligned}
        \sum_{n>N_1}\abs{e_n(q)}L^n
        &\le C_\rho q^2\sum_{n>N_1\text{, odd}}\left(\frac L\rho\right)^n
        =C_\rho q^2 \frac{(L/\rho)^{N_{1}+1}}{1-(L/\rho)^2}\\ 
        &\leq 3 C_\rho q^{2} .817^{N_1}
        \leq 3 C_\rho q^{2} q^{\frac{3.5\log(1/.817)}{2\log\log(1/q)}}
         \leq 3 C_\rho q^{2} q^{\frac{0.3537}{\log\log(1/q)}}
        =o(q^2).
        \end{aligned}
\end{equation}
% 3.5*log(1/.817) / 2

% sum_n>k p^k =p^k+1 + p^k+2 +...
%p sum_...  =         p^k+2 + ...
%  sum_n>k p^k = p^k+1 / (1-p)
%
% e^t = 1/q.  q = e^-t.  q^0.01 = e^-.01t
% .98^ 3t/2log t < q^.01 = e^-.01T ?
%for x = 4:30, x,.98^(3*x / log(x)) - exp(-.01*x), end
%
% L/R ^N1 < .9^N1 = .9 ^ 3t / 2log t
%  = .9 ^ 3 log(1/q) / 2loglog(1/q)
%  =  (1/q)^ 3 ln.9 / 2 loglog 1/q
%  =  q ^ 3 ln(1/.9) / 2loglog 1/q
%
%  q ^ 1/log log q
%  take log get log q / loglog q

It remains to bound the \(d_n\)-tail.  Since
\begin{equation}\label{qdelbd}
        \max_{n\le N_0}q\abs{\delta_n}
        \le C  q5^{N_0}
       %=C_\rho\exp\big(-T+(2.5\log 5)T/\log T)\big)=o(1),
       \leq C  q^{1-\frac{3.7\log 5}{\log\log (1/q)}}
       <1\text{, if }q<e^{-390}
\end{equation}
%
% need to adjust these constants a bit
%
% just need q delta_n < 1
%  C_r exp(-T + (2.5 ln 5)/log T)<1
% maybe just leave in terms of q?
%  5^N0 = exp (2.5 ln 5 ln (1/q) / loglog q)
%       = q^ 2.5 ln 5 / loglog q
%
% suffices to choose q^1-... < 1
% q^1-... < 1
% q<C_rho ^-1/(1-2.5log 5 / log log q)
%
% must have loglog(1/q) > 2.5 log 5 ~ 4.023
% i.e. q < exp(-56)
% if q<exp(-140), then loglog(1/q)>4.9416
% so 2.5 log5/loglog(1/q)<.815
%
% C_rho q^.23 < 1
% q < C_rho ^1/.23
%
% q^1 - 2.5 log 5 / loglogq<1
% 2.5 log 5 / log log q <1
% log log q > 2.5 log 5
% log q > 5^2.5
%
% use N1 = 3.7 ...  in the d=5 case
% or N1 = 3.4 ...  in the d=3 case
we have by \eqref{eq:delta-n-def} that \(s_nd_n(q)>0\) for all odd \(n\le N_0\) and for all \(q<e^{-390}\).  Therefore, due to the terms $N_1 <n \leq N_0$ being zero in the following sum,
\begin{equation}\label{eq:d-tail}
\begin{aligned}
        \sum_{\substack{n>N_1\\ n\text{ odd}}}
            \big(\abs{d_n(q)}-s_nd_n(q)\big)L^n
        &=\sum_{\substack{n>N_0\\ n\text{ odd}}}
            \big(\abs{d_n(q)}-s_nd_n(q)\big)L^n\\
        &\stackrel{\eqref{sdedef}}{\le} 2\sum_{n>N_0\text{, odd}}\big(\abs{b_n}+q\abs{c_n}\big)L^n  \\
        &\stackrel{\eqref{eq:bn-cn}}{\le} 2\sum_{n>N_0\text{, odd}}\frac{L^n}{n!}
          +\frac{2}{7}q\sum_{n>N_0\text{, odd}}\frac{(5L)^n}{n!}.
\end{aligned}
\end{equation}
% dn = bn+qcn

For every fixed \(\lambda>0\) and every sufficiently large \(N\),
\begin{equation}\label{eq:factorial-tail}
        \sum_{n>N}\frac{\lambda^n}{n!}
        \le 2\frac{\lambda^{N+1}}{(N+1)!}
        \le 2\left(\frac{e\lambda}{N+1}\right)^{N+1}.
\end{equation}
With \(N=N_0\), the right-hand side is \(\exp(-3.7T+o(T))=o(q^2)\), because
\(N_0\log N_0=3.7T+o(T)\).  More specifically,
$$
\sum_{n>N_0}\frac{L^n}{n!}
\leq q^{2.1},
\qquad
\sum_{n>N_0}\frac{(5L)^n}{n!}
<q^{1.01}\text{, if }q<e^{-390}
$$
So, the right side of \eqref{eq:d-tail} is \(o(q^2)\).
With \eqref{eq:error-tail-o}, \eqref{eq:difference-nonnegative}, \eqref{eq:split-tail}, this proves \eqref{eq:absolute-vs-alternating} when $d=5$.
%N>22 lambda=L=.88137...
% e lambda<2.4
% log (e lambda) = 1 - .1267 ~ .8737
% -N log N + N log(elambda)
% = -Nlog N -.8737 N < -N log N
% Nlog N
%=(5/2)(T/log T)(log(5/2) + log T - loglog T)
%=2.29T/logT + 2.5 T - 2.5Tloglog T/log T
% use loglog T / log T<1/3??
%
% can check  > 2.1T, so sum is <= exp(-2.1T)
% = q^2.1
%
%for T = 10:100, T,2.29*T/log(T) + 2.5*T - 2.5*T*log(log(T)) /log(T) - 2.1*T, end
%
%  ex.  q = e^-390.  T = 390
%   N = 3.7*(390)/log(390)~241
%   (elam / N)^2 ~ (2.4/241)^241
%    ~ 11110
%  241*log(2.4/241)  < q^2.1
% (elam5 / N)^2 ~ (12/241)^241
% %  241*log(12/241)  ~ 722
%
% meanwhile, lambda = 5*L
% log(e *5*lambda) = 2.483
% -N log N + N log(elambda)
% = -Nlog N + 2.483 N 
% negate it get
% =2.29T/logT + 2.5 T - 2.5Tloglog T/log T
%  - 2.483*2.5 T / log T
% = -3.1975T/logT + 2.5 T - 2.5Tloglog T/log T
% >.01T when T>118
\end{proof}
\begin{remark}\label{d3l5rk}
In the $d=3$ case, we let $N_0\colonequals\lfloor4 T/\log T\rfloor$, and \eqref{eq:delta-n-def} becomes
$$\delta_n = \frac{7-3^{n}}{2\sqrt{\pi}}.$$
% for n=20:30, n, -n + (3.5*log(3))*n/(2*log(n)) + .4*n, end
%
% C_rho exp(-T+(3log 3)t/(2lot t))
% < C_rho exp(-T/2)<1/2
% -t/2 < log(1/2crho)
% t> -2 log(1/2crho)
% log(1/q)>-2 log(1/2crho)
% 1/q >(1/2crho)^-2
% z< (2crho)^-2
with $|\delta_n|\leq C 3^n$ with $C=1/2$, and $-T+3.5(\log3 )T/(2\log T)<-.4T$ when $T>27$, 
so it suffices to choose $q<\min(e^{-27},2^{-2}C^{-2})$, 
%then \eqref{eq:error-tail-o} is bounded by $6C_\rho q^2 $
then \eqref{qdelbd} becomes
$$\max_{n\leq N_0}C  q 3^{N_0}
=C q^{1-\frac{4\log 3}{\log\log(1/q)}}
<1\text{, if }q<e^{-81},$$
Then \eqref{eq:d-tail} becomes
% need log log(1/q)> 4 log 3
%  log(1/q)> 3^4
$$
        \sum_{\substack{n>N_1\\ n\text{ odd}}}
            \big(\abs{d_n(q)}-s_nd_n(q)\big)L^n
        \leq
        2\sum_{n>N_0\text{, odd}}\frac{L^n}{n!}
          +q\sum_{n>N_0\text{, odd}}\frac{(3L)^n}{n!}.
$$
and the inequalities after that become
$$
\sum_{n>N_0}\frac{L^n}{n!}\leq q^{2.1},
\qquad
\sum_{n>N_0}\frac{(3L)^n}{n!}
<q^{1.01}\text{, if }q<e^{-65}
$$
% exp(exp(2.5*log(5)))
%
%  n>=48, have
%  1 - 2.5 log3 / (log n)< .29
% so q<C_rho ^-1/.29
%Crho ~ e^15.5, so q < e^-54
%
% meanwhile, lambda = 3*L=3*log(1+sqrt(2))
% log(e lambda) = 1.9723
% -N log N + N log(elambda)
% = -Nlog N + 1.973 N 
% negate it get
% =2.29T/logT + 2.5 T - 2.5Tloglog T/log T
%  - 1.973*2.5 T / log T
% = -2.6425 T/logT + 2.5 T - 2.5Tloglog T/log T
% >.01T when T>118
%
% for T=60:70, T, -2.6425*T/log(T) + 2.5* T - 2.5*T*log(log(T))/log(T) - 1.01*T, end
%
% ex.  q = e^-81
% N = 4 *81 / log(81) ~ 73
%  (e lam / N)^N ~ (2.4 / 73)^73
%  log of it is
%  ~ 250
%
%  %  (e 3 lam / N)^N ~ (7.2 / 73)^73
%  log of it is
%  ~ 169
\end{remark}

\section{Evaluation of \texorpdfstring{$B_q(iL)$}{Bq(iL)}}

Lemma \ref{lem:absolute-vs-alternating} reduces the problem of understanding $A_q(L)$ to instead understanding $B_q(iL)$.  In this section, we therefore write a power series expansion for $B_q(iL)$.  As we will see below, this reduces to understanding $H_{\eta}(i)$ or equivalently $\widetilde{H}_q(i)$.

For these reasons, the proof of Theorem 4.4 in \cite{braverman13} gives the following $d=5$ expansion, albeit without explicit error bound
\begin{equation}\label{eq:theorem44-expansion}
        4\pi\frac{H_\eta(i)}{i}
        =4\pi L+1600\sqrt2\,\eta^4+O(\eta^6).
\end{equation}
In Section \ref{secphi} and \eqref{phi5bd}, we show that the $O(\eta^6)$ term is bounded by $2.3\cdot 10^{13}|\eta|^6$.
% phi
In terms of \(q=\eta^2\), this is
\begin{equation}\label{eq:Hq-i-expansion}
        \frac{\widetilde H_q(i)}{i}\stackrel{\eqref{eq:normalization-factor}\wedge\eqref{eq:q-def}}{=}L+\frac{400\sqrt2}{\pi}q^2+O(q^3),
\end{equation}
and the $O(q^3)$ term can be taken to be $2\cdot 10^{12}q^3$.  Also, \eqref{eq:Q-def} gives \(Q(i)=0\), because \(i^5=i\), explaining the lack of linear term in \(q\) in \eqref{eq:Hq-i-expansion}.  Likewise when $d=3$, we obtain in \eqref{phi3bd}
\[
        \left|
        4\pi\frac{H_\eta(i)}{i}-4\pi L-48\sqrt2\,\eta^4
        \right|
        \le
        2.7\cdot10^7|\eta|^6 .
\qquad
        \left|
        \frac{\widetilde{H}_q(i)}{i}-L-\frac{12\sqrt{2}}{\pi}\,q^2
        \right|
        \le
        2.7\cdot10^7|q|^3 .
\]

\begin{lemma}\label{lem:Bq-iL}
Let $d=5$.  Then as \(q\downarrow0\), 
\begin{equation}\label{eq:Bq-iL}
        B_q(iL)=i-\frac{800i}{\pi}q^2+O(q^3),
\end{equation}
where $O(q^3)$ is bounded by $e^{9.4}q^3 + e^{36}q^4
+e^{59}q^{5}$.
%ln(852)+ 12 *ln(10) + 24.6
\end{lemma}

\begin{proof}
Recall that $\sin(iL)=i\sinh L\stackrel{\eqref{sldef}}{=}i$, so $\Psi(iL)\stackrel{\eqref{eq:Psi-def}}{=}0$.  Then 
Lemma~\ref{lem:inverse-expansion} applies at \(iL\) since $|iL|\stackrel{\eqref{sldef}}{<}.89$.  That is, if we define
\begin{equation}\label{dqdef}
        \delta_q\colonequals B_q(iL)-i.
\end{equation}

Then $|\delta_{q}|\leq q^{2}(q^{2}e^{63}+qe^{31}+e^{25})$ by Remark \ref{exrk}.  Using that $B_{q}$ is the inverse of $\widetilde{H}_q$ then Taylor expanding \(\widetilde H_q\) at \(i\), using analyticity on
\(S\), gives
\begin{equation}\label{eq:Taylor-at-i}
        iL=\widetilde{H}_{q}(B_{q}(iL))\stackrel{\eqref{dqdef}}{=}\widetilde H_q(i+\delta_q)
        =\widetilde H_q(i)+\widetilde H_q'(i)\delta_q+O(\delta_q^2).
\end{equation}
By Lemma~\ref{lem:Hq-expansion}, differentiated once in \(z\), (i.e. Remark \ref{ubrk}, using $\arcsin'(i)=1/\sqrt{2}$ and $Q'(i)=4\sqrt{2/\pi}$)
\[
        \widetilde H_q'(i)=\widetilde H_0'(i)+[4\sqrt{2/\pi}]q+O(q^2)=\frac{1}{\sqrt{2}}+q4\sqrt{2/\pi}+O(q^2),
\]
where $O(q^2)\leq 213q^{2}e^{24.6}$.
% Q' = -(4*z^6 - 5*z^4 + 1)/(1 - z^2)^(3/2)
% = 2 root 2 at z=i multipled by 2/sqrt(pi)
% so Q' = 4\sqrt{2/\pi}
Equation \eqref{eq:Hq-i-expansion} says
\[
        \widetilde H_q(i)=iL+i\frac{400\sqrt2}{\pi}q^2+O(q^3),
\]
where $O(q^3)$ is at most $2\cdot 10^{12}q^{3}$.  Substituting this and the $\delta_q = O(q^2)$ bound back into \eqref{eq:Taylor-at-i},
\[
        \delta_q
        =-\frac{i(400\sqrt2/\pi)q^2+O(q^3)}{1/\sqrt2+q4\sqrt{2/\pi}+O(q^2)}
        =-\frac{800i}{\pi}q^2
        +O(q^3),
\]
where $O(q^3)$ is at most 
$$11494q^3 + 76710 e^{24.6}q^4
+1.2\cdot10^{13}q^{4}
+852\cdot 10^{12}e^{24.6}q^{5},$$
% delta_q = A/1+epsilon
% A/1+epsilon =A + 1.1 A*epsilon 
%            =A(1+1.1*epsilon)
%1/(1-x)=1+x+x^2+...
% for x small, |1/(1-x) - 1|<1.1x
%deltaq -(-i800/piq2 + sqrt(2O(q^3)))
% <=8/rootpi q+ sqrt(2)O(q^2)
% so
% deltaq = -800i/pi q^2
% +1.01[800i/pi q^2+sqrt(2)O(q^3)]
%*[8/rootpi q+sqrt(2)O(q^2)] 
%
% 2 * 10^12 * sqrt(2)*8/sqrt(pi)
% 800*sqrt(2)*213*exp(24.6)/pi
which proves the lemma.
\end{proof}

Combining Lemma~\ref{lem:absolute-vs-alternating} with Lemma~\ref{lem:Bq-iL}, we obtain, for $d=5$,
\begin{equation}\label{eq:Aq-L}
        A_q(L)=1-\frac{800}{\pi}q^2+o(q^2)<1
\end{equation}
for all sufficiently small \(q>0\), i.e. $o(q^2)$ term is at most
$$
e^{9.4}q^3 + e^{36}q^4
+e^{59}q^{5}
+
6C_\rho q^{2} q^{\frac{0.3537}{\log\log(1/q)}}+q^{2.01},
$$
if 
$q<\min(e^{-390},C_\rho^{-1/.185})$.  That is, we found our desired upper bound on $A_q(L)$.

\begin{remark}\label{d3aqrk}
In the $d=3$ case, we have $|\delta_{q}|\leq q^{2}(q^{2}e^{42}+qe^{29}+e^{15.5})$ by Remark \ref{exrk}, $Q'(i)=2\sqrt{2/\pi}$, $\widetilde{H}_q (i)=iL+q^{2}i12\sqrt{2}/\pi+O(q^{3})$ with $O(q^3)$ bounded by $2.7\cdot 10^{7}|q|^3$, and the $\widetilde{H}_q'(i)$ expansion has error term bounded by $213 e^{15}$ by Remark \ref{ubrk}, so we get
%-(- 2*z^4 + 3*z^2 + 1)/(1 - z^2)^(3/2)
$$\delta_q =-\frac{i(12\sqrt2/\pi)q^2+O(q^3)}{1/\sqrt2+q2\sqrt{2/\pi}+O(q^2)}
        =-\frac{24i}{\pi}q^2
        +O(q^3),$$
with $O(q^3)$ term bounded by $13 q^3+ e^{24} q^4 + e^{39} q^5$.
% delta_q = A/1+epsilon
% A/1+epsilon =A + 1.1 A*epsilon 
%            =A(1+1.1*epsilon)
%1/(1-x)=1+x+x^2+...
% for x small, |1/(1-x) - 1|<1.1x
%deltaq -(-i12root2/piq2 + sqrt(2O(q^3)))
% <=4q+ sqrt(2)O(q^2)
% so
% deltaq = -800i/pi q^2
% +1.01[12root2/pi q^2+sqrt(2)O(q^3)]
%*[4/rootpi q+sqrt(2)O(q^2)] 
%
% = q^3 (48*sqrt(2)/pi^(3/2))
% + q^4 (4*sqrt(2/pi))*2.7*10^7
% + q^4 (24*213 e^15)
% + q^5 (2*213*exp(15) * 2.7*10^7)
%
% <= 13 q^3+ e^24 q^4 + e^39 q^5
Combining Lemma~\ref{lem:absolute-vs-alternating} with this analogue of Lemma~\ref{lem:Bq-iL}, we obtain, for $d=3$,
\begin{equation}\label{aqd3}
A_q(L)=1-\frac{24}{\pi}q^{2}+o(q^{2}),
\end{equation}
with $o(q^2)$ term bounded from Remark \ref{d3l5rk} by
$$13 q^3+ e^{24} q^4 + e^{39} q^5
+6C_\rho q^{2}q^{\frac{0.3537}{\log\log(1/q)}}+q^{2.01},\qquad\forall\,0<q<\min(e^{-81},C_\rho^{-1/.29})$$
\end{remark}

\section{Qualitative Proof Completion}\label{secqual}

Since we obtained our desired quantitative estimates that $A_q(L)<1$ for small $q$ (with $d=5$ in \eqref{eq:Aq-L} and $d=3$ in \eqref{aqd3}), we can now conclude the proof of Theorem \ref{mainthm}.  We will postpone the computation of explicit constants in Theorem \ref{mainthm} to Section \ref{secquant}.

\begin{proof}[Proof of Theorem \ref{mainthm}]
We first show the absolute coefficient sum $A_q$ exceeds \(1\) slightly to the right of
\(L\).  Since \(L=\operatorname{arsinh}(1)<0.9\), choose a radius
\(r\in(0.9,1)\).  Lemma~\ref{lem:inverse-expansion} gives, for all small \(q\), coefficient
bounds that are summable at \(t=0.9\).  Therefore, by dominated convergence,
\begin{equation}\label{eq:Aq-09}
        \lim_{q\downarrow0}A_q(0.9)
        =\sum_{\substack{n\ge1\\ n\text{ odd}}}\frac{0.9^n}{n!}
        =\sinh(0.9)>1.
\end{equation}
Together with \eqref{eq:Aq-L}, this implies that, for every sufficiently small \(q>0\),
there exists \(\gamma_q\in(L,0.9)\) such that
\begin{equation}\label{eq:gamma-condition}
        A_q(\gamma_q)=1.
\end{equation}
Indeed, \(A_q(t)\) is continuous on \([0,0.9]\), since its defining series converges
uniformly there.

It remains to translate \eqref{eq:gamma-condition} to the normalization in Definition \ref{kdef}.  By \eqref{eq:normalization-factor},
\[
        H_{f_{\sqrt q},f_{\sqrt q}}(z)=\frac2\pi\widetilde H_q(z),
\]
and hence the inverse in Definition \ref{kdef} is
\begin{equation}\label{eq:definition-inverse}
        H_{f_{\sqrt q},f_{\sqrt q}}^{-1}(z)
        =B_q\left(\frac\pi2z\right)
        =\sum_{\substack{n\ge1\\ n\text{ odd}}}a_n(q)\left(\frac\pi2\right)^nz^n.
\end{equation}
Thus by \eqref{eq:Aq-def}, \eqref{eq:gamma-condition} is exactly
\begin{equation}\label{pisum}
        \sum_{\substack{n\ge1\\ n\text{ odd}}}
        \left|a_n(q)\left(\frac\pi2\right)^n\right|
        \left(\frac{2\gamma_q}{\pi}\right)^n=1.
\end{equation}
This is condition \eqref{one10} for $H_{f_{\sqrt{q}},f_{\sqrt{q}}}$ with $\widehat{a}_n = a_n(q)(\pi/2)^n$ for all $n\geq1$ odd by \eqref{eq:definition-inverse}. Therefore, when $d=5$, \((f_{\sqrt q},f_{\sqrt q})\) is a Krivine rounding scheme with
\begin{equation}\label{cpi}
        c(f_{\sqrt q},f_{\sqrt q})\stackrel{\eqref{pisum}}{=}\frac{2\gamma_q}{\pi}
        >\frac2\pi L
        =\frac2\pi\log(1+\sqrt2).
\end{equation}
Since \(q=\eta^2\), Theorem~\ref{mainthm} follows.  Similarly, when $d=3$, we replace \eqref{eq:Aq-L} with \eqref{aqd3} and deduce $c(f_{\sqrt{q}}, f_{-\sqrt{q}})>\frac{2}{\pi}\log(1+\sqrt2)$ for sufficiently small $q>0$.

\end{proof}

\section{Quantitative Proof Completion}\label{secquant}

In this section, we make the argument from Section \ref{secqual} quantitative, in order to prove explicit upper bounds on $K_G$.  The Krivine parameter \(\gamma_q\) is the unique positive solution of
\[
        A_q(\gamma_q)=1.
\]
At \(q=0\), one has \(B_0(z)=\sin z\), and therefore
\[
        A_0(t)
        =
        \sum_{\substack{n\ge1\\ n\ {\rm odd}}}
        \frac{t^n}{n!}
        =
        \sinh t.
\]
Since \(L\stackrel{\eqref{sldef}}{=}\operatorname{arsinh}(1)\),
\[
        A_0(L)=1,
        \qquad
        A_0'(L)=\cosh L=\sqrt2.
\]

By the mean-value theorem, for some \(\xi_q\) between \(L\) and \(\gamma_q\),
\begin{equation}\label{gqeq}
        1-A_q(L)
        =
        A_q(\gamma_q)-A_q(L)
        =
        A_q'(\xi_q)(\gamma_q-L).
\end{equation}

\begin{proof}[Proof of \eqref{kg1}]
Let $d=5$.  Since $.88<\gamma_{q}<.9$, $.88<\xi_q<.9$, so Lemma \ref{aqplemma} with $\rho=1$ and $\cosh(L)=\sqrt{2}$ imply
\begin{flalign*}
|A_q'(\xi_q)-\sqrt{2}|
&\leq|A_q'(\xi_q)-\cosh(\xi_q)|
+|\cosh(\xi_q)-\cosh(L)|\\
&\leq 38q+q^2
[q^2 e^{63}+q e^{31} + e^{25}](100)
+.02.
\end{flalign*}
%cosh(.9) - cosh(log(1+sqrt(2)))

%Since \(\xi_q\to L\), we have
%\[
%        A_q'(\xi_q)\to \cosh L=\sqrt2.
%\]
Using \eqref{eq:Aq-L},
\[
        1-A_q(L)
        =
        \frac{800}{\pi}q^2+o(q^2),
\]
with $o(q^2)$ term bounded by $e^{9.4}q^3 + e^{36}q^4
+e^{59}q^{5}
+
6C_\rho q^{2} q^{\frac{0.3537}{\log\log(1/q)}}+q^{2.01}$.  Therefore
\[
        \gamma_q-L
        \stackrel{\eqref{gqeq}}{=}
        \frac{
        \frac{800}{\pi}q^2+o(q^2)
        }{
        \sqrt2+o(1)
        }
        \geq
        \frac{\frac{800}{\pi}q^2 -(e^{9.4}q^3 + e^{36}q^4
+e^{59}q^{5}
+
6C_\rho q^{2} q^{\frac{0.3537}{\log\log(1/q)}}+q^{2.01})}{\sqrt{2}+.02+38q+100q^{2}[q^2 e^{63}+q e^{31} + e^{25}]}
        .
\]
That is,
\[
        \gamma_q
        \geq
        L+
        q^{2}\frac{\frac{800}{\pi} -(e^{9.4}q + e^{36}q^2
+e^{59}q^{3}
+
6C_\rho q^{\frac{0.3537}{\log\log(1/q)}}+q^{0.01})}{\sqrt{2}+.02+38q+100q^{2}[q^2 e^{63}+q e^{31} + e^{25}]},
\]
if $q<e^{-390}$.  So, choosing $q=e^{-450}$, we get
\[
\gamma_{q}\geq \log(1+\sqrt{2})+e^{-900}\frac{250}{1.435}
\geq\log(1+\sqrt{2})+e^{-895}.
\]

% C_rho <= 2[q^{2}e^{63}+qe^{31}+e^{25}]

Corollary \ref{ckcor} then yields
\[
\boxed{K_{G}\leq\frac{1}{c(f_{\sqrt{q}},f_{\sqrt{q}})}\stackrel{\eqref{cpi}}{=}\frac{\pi}{2\gamma_q}
\leq\frac{\pi}{2(\log(1+\sqrt{2})+e^{-895})}
\leq\frac{\pi}{2\log(1+\sqrt{2})}-e^{-895}.
}
\]
%
%f(x) = (pi/2)(1/(log(1+sqrt(2))+x))
%f'(0)=-1/(log(1+sqrt(2)))^2~-1.2873
%
% exp(-895) = 10^-x.  x = 895/log 10
%\snote{improvement is $10^{-389}$}

\end{proof}
\begin{proof}[Proof of \eqref{kg2}]
Let $d=3$.  In this case, we have by Lemma \ref{aqplemma}
\begin{flalign*}
|A_q'(\xi_q)-\sqrt{2}|
&\leq|A_q'(\xi_q)-\cosh(\xi_q)|
+|\cosh(\xi_q)-\cosh(L)|\\
&\leq 12q+q^2
[q^{2}e^{42}+qe^{29}+e^{15.5}](100)
+.02.
\end{flalign*}
From Remark \ref{d3aqrk},
$$1-A_q(L)=\frac{24}{\pi}q^{2}+o(q^{2}),$$
with $o(q^2)$ term bounded by
$$13 q^3+ e^{24} q^4 + e^{39} q^5
+6C_\rho q^{2}q^{\frac{0.3537}{\log\log(1/q)}}+q^{2.01},\qquad\forall\,0<q<\min(e^{-81},C_\rho^{-1/.29}).$$

Therefore
\[
        \gamma_q-L
        \stackrel{\eqref{gqeq}}{=}
        \frac{
        \frac{24}{\pi}q^2+o(q^2)
        }{
        \sqrt2+o(1)
        }
        \geq
        \frac{\frac{24}{\pi}q^2 -(13 q^3+ e^{24} q^4 + e^{39} q^5
+6C_\rho q^{2}q^{\frac{0.3537}{\log\log(1/q)}}+q^{2.01})}{\sqrt{2}+.02+12q+100q^{2}[q^{2}e^{42}+qe^{29}+e^{15.5}]}
        .
\]
That is,
\[
        \gamma_q
        \geq
        L+
        q^{2}\frac{\frac{24}{\pi} -(13 q+ e^{24} q^2 + e^{39} q^3
+6C_\rho q^{\frac{0.3537}{\log\log(1/q)}}+q^{0.01})}{\sqrt{2}+.02+12q+100q^{2}[q^{2}e^{42}+qe^{29}+e^{15.5}]}.
\]$q<e^{-81}$.  So, choosing $q=e^{-250}$, we get
%24/pi - exp(-2.5)-6*exp(15.5)*exp(-250*0.3537/(log(250)))
%>3.97
% sqrt(2)+.02
%3.97 / (sqrt(2)+.02)=2.768
$$
\gamma_q\geq\log(1+\sqrt{2})+2.768 e^{-500}.
$$
Corollary \ref{ckcor} then yields
\[
\boxed{
K_{G}\leq\frac{1}{c(f_{\sqrt{q}},f_{-\sqrt{q}})}\stackrel{\eqref{cpi}}{=}\frac{\pi}{2\gamma_q}
\leq\frac{\pi}{2(\log(1+\sqrt{2})+2.768e^{-500})}
\leq\frac{\pi}{2\log(1+\sqrt{2})}-3.56e^{-500}.
}
\]
%
%f(x) = (pi/2)(1/(log(1+sqrt(2))+x))
%f'(0)=-1/(log(1+sqrt(2)))^2~-1.2873
%
% 6.56 e^-130 = 10^-x
%  x = (log(3.56) -500) / log(10)

%10^{-217}
\end{proof}

\section{Absolute Inverse Bound}

\begin{lemma}[Quantitative control of \(A_q'\)]\label{aqplemma}
% Let \(L=\log(1+\sqrt2)=\operatorname{arsinh}(1)\), and let
% \(L<\rho<R\).  Suppose that
% \[
% B_q(z)=\sin z+q\Psi_d(z)+E_q(z)
% \]
% on \(|z|\le R\), with
% \[
% \sup_{|z|\le R}|E_q(z)|\le C_Bq^2.
% \]
% Write
% \[
% B_q(z)=\sum_{\substack{n\ge1\\ n\ {\rm odd}}}a_n(q)z^n,
% \qquad
% A_q(t)=\sum_{\substack{n\ge1\\ n\ {\rm odd}}}|a_n(q)|t^n.
% \]
For any \(0\le t< \rho\),
\[
|A_q'(t)-\cosh t|
\le
qM_{d,1}(\rho)
+
q^2
[q^2 e^{63}+q e^{31} + e^{25}]
\frac{\rho}{(\rho-t)^2}\text{, if }d=5
\]
\[
|A_q'(t)-\cosh t|
\le
qM_{d,1}(\rho)
+
q^2
[q^{2}e^{42}+qe^{29}+e^{15.5}]
\frac{\rho}{(\rho-t)^2}\text{, if }d=3
\]
Here
%For \(d=3\),
\[
% M_{3,0}(\rho)
% =
% \frac{7\sinh\rho+\sinh(3\rho)}{2\sqrt\pi},
% \qquad
M_{3,1}(\rho)
=
\frac{7\cosh\rho+3\cosh(3\rho)}{2\sqrt\pi},
% For \(d=5\),
% \[
% M_{5,0}(\rho)
% =
% \frac{6\sinh\rho+5\sinh(3\rho)+\sinh(5\rho)}{8\sqrt\pi},
% \]
% and
% \[
\qquad
M_{5,1}(\rho)
=
\frac{6\cosh\rho+15\cosh(3\rho)+5\cosh(5\rho)}{8\sqrt\pi}.
\]
% (6*cosh(1)+15*cosh(3)+5*cosh(5))/(8*sqrt(pi))
%(7*cosh(1)+3*cosh(3))/(2*sqrt(pi))

% Define
% \[
% \varepsilon_{d,0}
% =
% qM_{d,0}(\rho)
% +
% C_Bq^2\frac{\rho}{R-\rho},
% \]
% and
% \[
% \varepsilon_{d,1}
% =
% qM_{d,1}(\rho)
% +
% C_Bq^2\frac{R}{(R-\rho)^2}.
% \]
% If
% \[
% \varepsilon_{d,0}<\sinh\rho-1,
% \]
% then the root \(\gamma_q\) of \(A_q(\gamma_q)=1\) lies in \([0,\rho]\), and
% \[
% \operatorname{arsinh}(1-\varepsilon_{d,0})
% \le
% \gamma_q
% \le
% \operatorname{arsinh}(1+\varepsilon_{d,0}).
% \]
% In particular,
% \[
% |\gamma_q-L|\le \varepsilon_{d,0}.
% \]
% Furthermore,
% \[
% |A_q'(\gamma_q)-\sqrt2|
% \le
% \varepsilon_{d,1}+\sinh(\rho)\varepsilon_{d,0}.
% \]
\end{lemma}

\begin{proof}
% Write
% \[
% \sin z+q\Psi_d(z)
% =
% \sum_{\substack{n\ge1\\ n\ {\rm odd}}}d_n(q)z^n,
% \qquad
% E_q(z)
% =
% \sum_{\substack{n\ge1\\ n\ {\rm odd}}}e_n(q)z^n.
% \]
%
% dn = bn+qcn
Define $d_{n},e_{n}$ as in \eqref{sdedef} and Lemma \ref{lem:inverse-expansion}.  Then
\[
a_n(q)=d_n(q)+e_n(q)
=b_n + qc_n +e_n(q).
\]
Since
\[
\bigl||a_n(q)|-|b_n|\bigr|
\le |a_n(q)-b_n|
\le q|c_n|+|e_n(q)|,
\]
where \(b_n\) is the coefficient of \(\sin z\) and \(c_n\) is the coefficient of
\(\Psi_d\), we get
\[
|A_q(t)-\sinh t|
\le
q\sum_{\substack{n\ge1\\ n\ {\rm odd}}}|c_n|t^n
+
\sum_{\substack{n\ge1\\ n\ {\rm odd}}}|e_n(q)|t^n.
\]
% The first sum is bounded by \(M_{d,0}(\rho)\) for \(t\le\rho\) by \eqref{eq:Psi-def} and \eqref{eq:bn-cn}.  For the
% second, \eqref{eq:coefficient-error} gives
% \[
% |e_n(q)|\le C_\rho q^2\rho^{-n}
% \]
% and therefore
% \[
% \sum_{n\ge1}|e_n(q)|t^n
% \le
% C_\rho q^{2}
% \sum_{n\ge1}\left(\frac{t}{\rho}\right)^n
% \le
% C_\rho q^2\frac{t}{\rho-t}.
% \]
% % t/rho / 1- t/rho  = t / rho - t
% This proves the bound for \(A_q(t)-\sinh t\) since Remark \ref{exrk} says $C_\rho \leq 2[q^2 e^{63}+q e^{31} + e^{25}]$ for $d=5$ and $\leq q^{2}e^{42}+qe^{29}+e^{15.5}$ for $d=3$.

We similarly bound $A_q'$ as follows.
\[
|A_q'(t)-\cosh t|
\le
q\sum_{\substack{n\ge1\\ n\ {\rm odd}}}n|c_n|t^{n-1}
+
\sum_{\substack{n\ge1\\ n\ {\rm odd}}}n|e_n(q)|t^{n-1}.
\]
The first sum is bounded by \(M_{d,1}(\rho)\), and the second satisfies
\[
\sum_{n\ge1}n|e_n(q)|t^{n-1}
\le
C_\rho q^2
\sum_{n\ge1}n\rho^{-n}t^{n-1}
=
C_\rho q^2\frac{\rho}{(\rho-t)^2}
\]
This proves the derivative bound, via \eqref{eq:Psi-def}, \eqref{eq:bn-cn}, \eqref{eq:coefficient-error} and Remark \ref{exrk}.

% Now assume \(\varepsilon_{d,0}<\sinh\rho-1\).  Since
% \[
% A_q(\rho)\ge \sinh\rho-\varepsilon_{d,0}>1,
% \]
% and \(A_q(0)=0\), the root \(\gamma_q\) lies in \([0,\rho]\).  At this root,
% \[
% A_q(\gamma_q)=1,
% \]
% so
% \[
% |\sinh\gamma_q-1|
% =
% |\sinh\gamma_q-A_q(\gamma_q)|
% \le
% \varepsilon_{d,0}.
% \]
% Therefore
% \[
% \operatorname{arsinh}(1-\varepsilon_{d,0})
% \le
% \gamma_q
% \le
% \operatorname{arsinh}(1+\varepsilon_{d,0}).
% \]
% Since \(\operatorname{arsinh}'(x)\le1\) for \(x\ge0\), and
% \(L=\operatorname{arsinh}(1)\), this implies
% \[
% |\gamma_q-L|\le \varepsilon_{d,0}.
% \]
% Finally,
% \[
% |A_q'(\gamma_q)-\sqrt2|
% \le
% |A_q'(\gamma_q)-\cosh\gamma_q|
% +
% |\cosh\gamma_q-\cosh L|.
% \]
% The first term is at most \(\varepsilon_{d,1}\), while the second is at most
% \[
% \sinh(\rho)|\gamma_q-L|
% \le
% \sinh(\rho)\varepsilon_{d,0},
% \]
% because \(\gamma_q,L\in[0,\rho]\).  This proves the final estimate.
\end{proof}

\section{Numerical Bounds}\label{app1}

This section records the numerical input used to make explicit estimates of $C_R$.  The main results are \eqref{eq:H-remainder-product-final} together with the numerically verified \eqref{eq:product-R108-final}, \eqref{eq:product-R108-final2}, \eqref{eq:product-R108-final3}, \eqref{eq:product-R108-final4}.  For any $z\in\C$, denote $\mathrm{erf}(z)\colonequals\frac{2}{\sqrt{\pi}}\int_{0}^{z}e^{-t^{2}}dt$ as the line integral in $\C$.  For any $0<R\leq1.1$, let
\begin{equation}\label{krdef}
        K_R\colonequals\{\sin z\colon\ |z|\le R\}\subset S\colonequals\{z\in\C\colon|\Re(z)|<1\}.
\end{equation}
For any \(w\in K_R\), write
\[
        T_w\colonequals\frac1{1-w^2},
        \qquad
        \Phi_w(a,b)\colonequals T_w(a^2+b^2-2wab),\qquad\forall\,a,b\in\R,
\]
and define $p_{z}(u,v)$ as in \eqref{eq:pz}.  Let \(G_z(a,b)\) be defined as in \eqref{gzdef}.  For mixed derivatives,
\[
        \partial_a^r\partial_b^sG_w(a,b)
        =4\partial_a^{r-1}\partial_b^{s-1}p_w(a,b),\qquad\forall\,r,s\geq1.
\]
Thus, for example,
\[
\partial_a^3\partial_bG_w
        =4(\Phi_a^2-\Phi_{aa})p_w,\qquad
        \partial_a^2\partial_b^2G_w
        =4(\Phi_a\Phi_b-\Phi_{ab})p_w,\qquad
        \partial_a\partial_b^3G_w
        =4(\Phi_b^2-\Phi_{bb})p_w.
\]
The pure fourth derivatives are evaluated from the one-dimensional formula
\begin{equation}\label{eq:one-dimensional-pure-derivative-input}
        \int_{\mathbb R}\operatorname{sgn}(v-b)p_w(a,v)\,dv
        =\frac{e^{-a^2}}{\sqrt\pi}\operatorname{erf}
        \!\left(\frac{wa-b}{\sqrt{1-w^2}}\right),
\end{equation}
which implies
\begin{equation}\label{eq:Ga-formula-for-remainder}
        \partial_aG_w(a,b)
        =-\frac{2e^{-a^2}}{\sqrt\pi}
        \operatorname{erf}\!\left(\frac{wa-b}{\sqrt{1-w^2}}\right).
\end{equation}
Put
\[
        s\colonequals1-w^2,
        \qquad
        \lambda\colonequals\frac{w}{\sqrt{s}},
        \qquad
        \tau_a\colonequals\frac{wa-b}{\sqrt{s}},
        \qquad
        \tau_b\colonequals\frac{wb-a}{\sqrt{s}}.
\]
Differentiating \eqref{eq:Ga-formula-for-remainder} three more times gives the explicit
pure fourth derivative
\begin{equation}\label{eq:Gaaaa-explicit-remainder}
\begin{aligned}
        \partial_a^4G_w(a,b)
        &=-\frac{2e^{-a^2}}{\sqrt\pi}
        \Bigg[
        (12a-8a^3)\operatorname{erf}(\tau_a)  \\
        &\qquad\qquad
        +\frac{e^{-\tau_a^2}}{\sqrt\pi}
        \left(
        6\lambda(4a^2-2)+24a\lambda^2\tau_a
        +4\lambda^3(2\tau_a^2-1)
        \right)
        \Bigg].
\end{aligned}
\end{equation}
The formula for \(\partial_b^4G_w(a,b)\) is obtained from
\eqref{eq:Gaaaa-explicit-remainder} by interchanging \(a\) and \(b\), i.e. by replacing
\(a,\tau_a\) with \(b,\tau_b\).  The mixed fourth derivatives have the simpler density
forms
\begin{equation}\label{eq:mixed-fourth-explicit-remainder}
\begin{aligned}
        \partial_a^3\partial_bG_w(a,b)
        &=4p_w(a,b)
        \left[\frac{4(a-wb)^2}{s^2}-\frac{2}{s}\right], \\
        \partial_a^2\partial_b^2G_w(a,b)
        &=4p_w(a,b)
        \left[\frac{4(a-wb)(b-wa)}{s^2}+\frac{2w}{s}\right], \\
        \partial_a\partial_b^3G_w(a,b)
        &=4p_w(a,b)
        \left[\frac{4(b-wa)^2}{s^2}-\frac{2}{s}\right].
\end{aligned}
\end{equation}
Equations \eqref{eq:Gaaaa-explicit-remainder} and
\eqref{eq:mixed-fourth-explicit-remainder} are the formulas used in the numerical
supremum computations below.

We now record the precise fourth-order Taylor remainder that is used to bound the
\(O(q^2)\) term.  Fix real numbers \(A,B\), and define the one-variable function
\[
        \phi(t)=G_w(tA,tB),\qquad 0\le t\le \sqrt q.
\]
The integral Taylor formula gives, for complex-valued \(G_w\),
\begin{equation}\label{eq:G-one-variable-Taylor-remainder}
\begin{aligned}
        G_w(\sqrt q A,\sqrt q B)
        &=G_w(0,0)
        +\sqrt q\,\phi'(0)
        +\frac q2\phi''(0)
        +\frac{q^{3/2}}6\phi'''(0) \\
        &\quad
        +\frac{q^2}{6}\int_0^1(1-\theta)^3
        \phi^{(4)}(\theta\sqrt q)\,d\theta.
\end{aligned}
\end{equation}
Here
\begin{equation}\label{eq:directional-fourth-operator}
\begin{aligned}
        \phi^{(4)}(t)
        &=\mathcal D_{4,w}(tA,tB;A,B), \\
        \mathcal D_{4,w}(a,b;A,B)
        &=A^4G_{aaaa}(a,b)
        +4A^3B G_{aaab}(a,b)
        +6A^2B^2G_{aabb}(a,b)  \\
        &\quad
        +4AB^3G_{abbb}(a,b)+B^4G_{bbbb}(a,b).
\end{aligned}
\end{equation}
%This is the exact fourth-order directional derivative in the threshold direction
%\((A,B)\).  
%Notice that we use the integral remainder rather than a Lagrange remainder:
%for complex \(w\), the function \(t\mapsto G_w(tA,tB)\) is complex-valued, whereas the
%integral formula is valid componentwise without choosing a special intermediate point.

Let \(d\in\{3,5\}\), let \(\sigma\in\{-1,1\}\), and set
\[
        A=h_d(u),\qquad B=\sigma h_d(v).
\]
Write \(\widetilde H_q^{(d,\sigma)}\) for the analogue of \(\widetilde H_q\) obtained from
these two perturbations, and write
\begin{equation}\label{eq:Q-d-sigma-remainder-section}
        Q_{d,\sigma}(w)=
        \frac{2(\sigma w^d-w)}{\sqrt\pi\sqrt{1-w^2}}.
\end{equation}
Thus the symmetric construction is \((d,\sigma)=(5,1)\), while the asymmetric
\(h_3,-h_3\) construction is \((d,\sigma)=(3,-1)\).
After multiplying \eqref{eq:G-one-variable-Taylor-remainder} by
\((\pi/2)p_w(u,v)\) and integrating in \((u,v)\), the linear term vanishes and the cubic
term vanishes by antisymmetry.  The quadratic term is the already computed
\(qQ_{d,\sigma}(w)\).  Therefore
\begin{equation}\label{eq:Hq-exact-fourth-remainder-integral}
\begin{aligned}
        \widetilde H^{(d,\sigma)}_q(w)
        &\stackrel{\eqref{eq:directional-fourth-operator}}{=}\arcsin w+qQ_{d,\sigma}(w)  \\
        &\quad
        +\frac{\pi q^2}{12}
        \int_0^1(1-\theta)^3
        \iint_{\mathbb R^2}
        \mathcal D_{4,w}\bigl(\theta\sqrt q A,\theta\sqrt q B;A,B\bigr)
        p_w(u,v)\,du\,dv\,d\theta.
\end{aligned}
\end{equation}
Consequently the desired fourth-order bound is
\begin{equation}\label{eq:sharp-fourth-remainder-bound}
        \left|\widetilde H^{(d,\sigma)}_q(w)-\arcsin w-qQ_{d,\sigma}(w)\right|
        \le
        \frac{\pi q^2}{48}\Lambda_{d,\sigma}(w),
\end{equation}
where
\begin{equation}\label{eq:Lambda-sharp-definition}
        \Lambda_{d,\sigma}(w)=
        \iint_{\mathbb R^2}
        \sup_{a,b\in\mathbb R}
        \left|
        \mathcal D_{4,w}\bigl(a,b;h_d(u),\sigma h_d(v)\bigr)
        \right|
        |p_w(u,v)|\,du\,dv.
\end{equation}
%This is the sharper version of the bound: the supremum is taken over the actual
%fourth-order Taylor expression, not over each partial derivative separately.

%For comparison with the constants used below,
We further bound
\(\Lambda_{d,\sigma}\) by derivative-by-derivative suprema.  Define for any $w\in S$
\begin{equation}\label{eq:mj-Jj-definitions}
\begin{aligned}
        m_{j,d}(w)&\colonequals
        \sup_{a,b\in\mathbb R}
        \left|\partial_a^{4-j}\partial_b^jG_w(a,b)\right|,
        \qquad 0\le j\le4, \\
        J_{j,d}(w)&\colonequals
        \iint_{\mathbb R^2}
        |h_d(u)|^{4-j}|h_d(v)|^j |p_w(u,v)|\,du\,dv.
\end{aligned}
\end{equation}
Then for all $w\in S$
\begin{equation}\label{eq:componentwise-fourth-remainder-bound}
        \Lambda_{d,\sigma}(w)
        \stackrel{\eqref{eq:Lambda-sharp-definition}}{\le}
        \sum_{j=0}^4\binom4j m_{j,d}(w)J_{j,d}(w).
\end{equation}
%The still cruder estimate used in the first explicit extraction is obtained by setting
Further, set
%\(M_{4,d}(w)\colonequals\max_{0\leq j\leq 4}m_{j,d}(w)\) and
\[
        \mathcal J_d(w)\colonequals
        \iint_{\mathbb R^2}(|h_d(u)|+|h_d(v)|)^4|p_w(u,v)|\,du\,dv,
\]
which gives
\begin{equation}\label{eq:old-M4J-bound-from-sharp}
        \Lambda_{d,\sigma}(w)\le M_{4,d}(w)\mathcal J_d(w),
\end{equation}

where for any $w\in S$
\[
        M_{4,d}(w)\colonequals\sup_{a,b\in\mathbb R}\max_{r+s=4}
        \abs{\partial_a^r\partial_b^sG_w(a,b)},
\]
% \[
%         \mathcal J_d(w)\colonequals
%         \iint_{\mathbb R^2}
%         \left(\abs{h_d(u)}+\abs{h_d(v)}\right)^4
%         \abs{p_w(u,v)}\,du\,dv.
% \]
The fourth-order Taylor formula \eqref{eq:sharp-fourth-remainder-bound} gives: for any $w\in S$
\begin{equation}\label{eq:H-remainder-product-final}
\boxed{
        \abs{\widetilde H_q(w)-\arcsin w-qQ(w)}
        \stackrel{\eqref{eq:sharp-fourth-remainder-bound}\wedge\eqref{eq:old-M4J-bound-from-sharp}}{\le}
        q^2\frac{\pi}{48}M_{4,d}(w)\mathcal J_d(w),
        }
\end{equation}
where $\sigma\colonequals -1$ if $d=3$, $\sigma\colonequals 1$ if $d=5$, and 
\[
        Q(w)=Q_d(w)\colonequals\frac{2(\sigma w^d-w)}{\sqrt\pi\sqrt{1-w^2}}.
\]

Direct numerical maximization (available at Link \ref{oururl}) gives
\begin{equation}\label{eq:product-R108-final}
        \sup_{w\in K_{1.08}}
        \log\left(\frac{\pi}{48}M_{4,d}(w)\mathcal J_d(w)\right)
        <23.58,\text{ if }d=5
\end{equation}
\begin{equation}\label{eq:product-R108-final2}
        \sup_{w\in K_{1.09}}
        \log\left(\frac{\pi}{48}M_{4,d}(w)\mathcal J_d(w)\right)
        <24.58,\text{ if }d=5
\end{equation}
\begin{equation}\label{eq:product-R108-final3}
        \sup_{w\in K_{1.08}}
        \log\left(\frac{\pi}{48}M_{4,d}(w)\mathcal J_d(w)\right)
        <14.5,\text{ if }d=3
\end{equation}
\begin{equation}\label{eq:product-R108-final4}
        \sup_{w\in K_{1.09}}
        \log\left(\frac{\pi}{48}M_{4,d}(w)\mathcal J_d(w)\right)
        <15,\text{ if }d=3
\end{equation}

\section{Derivative Bounds}\label{app2}

We shall also need the differentiated version of the error bound \eqref{eq:H-remainder-product-final}.  Let \(0<R_-<R_+<1.1\), and define
\[
\mathcal E_q(\zeta)
=
\widetilde H_q(\sin\zeta)-\zeta-qQ(\sin\zeta).
\]
Assume that
\[
\sup_{|\zeta|\le R_+}|\mathcal E_q(\zeta)|
\le C_{R_+}q^2 .
\]
Since \(\mathcal E_q\) is holomorphic on \(|\zeta|<R_+\), Cauchy's estimate gives
\[
\sup_{|\zeta|\le R_-}|\mathcal E_q'(\zeta)|
\le
\frac{C_{R_+}}{R_+-R_-}q^2 .
\]
Equivalently,
\[
\sup_{|\zeta|\le R_-}
\left|
(\widetilde H_q\circ\sin)'(\zeta)
-
1
-
q(Q\circ\sin)'(\zeta)
\right|
\le
\frac{C_{R_+}}{R_+-R_-}q^2 .
\]
If one wants the derivative in the original variable \(w=\sin\zeta\), then
\[
\mathcal E_q'(\zeta)
=
\left[
\widetilde H_q'(w)-(\arcsin)'(w)-qQ'(w)
\right]\cos\zeta.
\]
Since
\[
\min_{|\zeta|\le R_-}|\cos\zeta|=\cos R_-,
\]
we therefore obtain the differentiated version of \eqref{eq:H-remainder-product-final}:
\begin{equation}\label{hqpbound}
\boxed{
\sup_{w\in K_{R_-}}
\left|
\widetilde H_q'(w)-(\arcsin)'(w)-qQ'(w)
\right|
\le
\frac{C_{R_+}}{(R_+-R_-)\cos R_-}q^2 ,
}
\end{equation}
%where \(K_{R_-}=\{\sin\zeta:\ |\zeta|\le R_-\}\).

\section{Lower Lipschitz bound for arcsin}

\begin{lemma}[Explicit lower Lipschitz bound for $\arcsin$]\label{lemma5}
Let \(0<r<s<1.1\), define $K_t$ as in \eqref{krdef}.  For any \(w_0\in K_r\) and $w\in K_{s}$
we have
\begin{equation}\label{asincon}
|H_0(w)-H_0(w_0)|
\ge
\frac1{\cosh s}|w-w_0|.
\end{equation}
Inequality \eqref{asincon} also holds under the following assumptions.  Let
\[
a\colonequals\frac{r+s}{2},
\qquad
\rho_{r,s}
\colonequals
\frac{s-r}{2}\cos a\,\frac{\sin a}{a}.
\]
If $w_0\in K_r$ and $|w-w_0|\leq\rho_{r,s}$, then $w\in K_s$, so \eqref{asincon} holds.
\end{lemma}

\begin{proof}
Let \(w_0=\sin\zeta_0\in K_r\) and \(w=\sin\zeta\in \sin(s\mathbb D)\).
Then \(H_0(w_0)=\zeta_0\) and \(H_0(w)=\zeta\) and
%.  Since \(s\mathbb D\) is
%convex,
\[
\sin\zeta-\sin\zeta_0
=
(\zeta-\zeta_0)
\int_0^1
\cos\bigl((1-t)\zeta_0+t\zeta\bigr)\,dt.
\]
For any $\xi\in\C$ with \(|\xi|\le s\),
\[
|\cos \xi|^2
=
\cos^2(\Re\xi)+\sinh^2(\Im\xi)
\le
1+\sinh^2s
=
\cosh^2s.
\]
% cos u = cos x cos iy -sin x sin iy
%       = cosx coshy - isinx sinhy
% cos^2x cosh^2y  + sin^2x sinh^2y
% = cos^2 x  + sinh^2 y
Hence
\[
|w-w_0|
\le
\cosh s\,|\zeta-\zeta_0|
=
\cosh s\,|H_0(w)-H_0(w_0)|.
\]
That is, 
\[
|H_0(w)-H_0(w_0)|
\ge
\frac1{\cosh s}|w-w_0|.
\]

Therefore, \eqref{asincon} holds.  Now, since \(s<\pi/2\), the map \(\sin\) is injective on
\(\overline{s\mathbb D}\).  Indeed, if
\(\sin\zeta=\sin\xi\), then
\[
2\cos\left(\frac{\zeta+\xi}{2}\right)
\sin\left(\frac{\zeta-\xi}{2}\right)=0.
\]
But \(|(\zeta+\xi)/2|\le s<\pi/2\), so the cosine factor cannot vanish,
and \(|(\zeta-\xi)/2|\le s<\pi/2\), so the sine factor vanishes only if
\(\zeta=\xi\).

We first lower-bound the distance from \(K_r\) to
\(\partial\sin(s\mathbb D)\).  Let
\[
|\zeta_0|\le r,\qquad |\zeta_*|=s.
\]
Then
\[
\sin \zeta_*-\sin\zeta_0
=
2\cos\left(\frac{\zeta_*+\zeta_0}{2}\right)
\sin\left(\frac{\zeta_*-\zeta_0}{2}\right).
\]
% sin(x-y)=sinx cosy - cosx sin y
% cos(x+y)=cosx cosy - sinx sin y
%
% multiply out get
%
%cos^2(y/2) sin(x)/2
%-sin^2(x/2) sin(y)/2
%-cos^2(x/2) sin(y)/2
%+sin^2(y/2)sin(x)/2
%=sin x - sin y   /2
Since
\[
\left|\frac{\zeta_*+\zeta_0}{2}\right|\le a,
\]
we have
\[
\left|
\cos\left(\frac{\zeta_*+\zeta_0}{2}\right)
\right|
\ge \cos a.
\]
Also, if \(|u|\le a<\pi/2\), then
\[
|\sin u|\ge \frac{\sin a}{a}|u|.
\]
Indeed, writing \(u=x+iy\), and using $\sin u = \sin x \cosh y + i\cos x \sinh y$, together with $\cosh^2 y = 1+\sinh^2 y$, $\sin x/x$ is decreasing for $x>0$ and $\sinh y\geq y$ for $y>0$,
% sin u = sinx cosh y + i cosx sinhy
%|| = sin^2x cosh^2 y + cos^2x sinh^2 y
% cosh^2 = 1 + sinh^2, so 
%|| = sin^2 + sinh^2 y
\[
|\sin u|^2
=\sin^2x+\sinh^2y
\ge
\left(\frac{\sin a}{a}\right)^2(x^2+y^2).
\]
% sinx / x is monotone decreasing, so
% (sin x /x) >= sin a /a
% sinh y >= y >= y sin a/a
Applying this to \(u=(\zeta_*-\zeta_0)/2\), and using
\[
|\zeta_*-\zeta_0|\ge s-r,
\]
gives
\[
|\sin \zeta_*-\sin\zeta_0|
\ge
(s-r)\cos a\,\frac{\sin a}{a}.
\]
Thus every point within distance \(\rho_{r,s}\) of \(K_r\) lies inside
\(\sin(s\mathbb D)\).

\end{proof}

\section{Polynomial Bounds}\label{secmore}

In this section, we record some numerical bounds for the $Q$ and $Q'$ terms appearing in \eqref{eq:H-remainder-product-final} and \eqref{hqpbound}.  Let $
r\colonequals\frac{43}{40}.
$  All bounds in this section are on the disk \(|z|\le r\).  Write $z=x+iy$.  Since \(r<\pi/2\), the disk \(\{|z|\le r\}\) contains no zero of
\(\cos z\).  Hence the functions below are holomorphic on a neighborhood of
\(\{|z|\le r\}\), and the maximum modulus principle reduces the estimates to
the boundary \(|z|=r\).

For \(d=3\), we use
\[
\Psi_3(z)
=
\frac{2}{\sqrt\pi}\bigl(\sin z+\sin^3z\bigr)
=
\frac{7\sin z-\sin(3z)}{2\sqrt\pi},
\]
and
\[
Q_3(w)
=
-\frac{2(w+w^3)}{\sqrt\pi\sqrt{1-w^2}}.
\]
Since, on \(|z|<\pi/2\), the chosen branch satisfies
\[
\sqrt{1-\sin^2 z}=\cos z,
\]
we have
\[
Q_3'(\sin z)
=
-\frac{2}{\sqrt\pi}
\frac{1+3\sin^2z-2\sin^4z}{\cos^3z}.
\]
Also
\[
(\arcsin)''(\sin z)
=
\frac{\sin z}{\cos^3z}.
\]

For \(d=5\), we use
\[
\Psi_5(z)
=
\frac{2}{\sqrt\pi}\bigl(\sin z-\sin^5z\bigr)
=
\frac{6\sin z+5\sin(3z)-\sin(5z)}{8\sqrt\pi},
\]
and
\[
Q_5(w)
=
\frac{2(w^5-w)}{\sqrt\pi\sqrt{1-w^2}}.
\]
Thus
\[
Q_5'(\sin z)
=
\frac{2}{\sqrt\pi}
\frac{-1+5\sin^4z-4\sin^6z}{\cos^3z}.
\]

We numerically verify (with details at Link \ref{oururl}) that
\begin{equation}\label{fif1}
\sup_{|z|\leq 1.075}|\Psi_3(z)|\leq2.22.
\end{equation}
\begin{equation}\label{fif2}\sup_{|z|\leq 1.075}|\Psi_5(z)|\leq4.1.
\end{equation}
\begin{equation}\label{fif3}\sup_{|z|\leq 1.075}|(\arcsin)''(\sin z)|\leq8.2.\end{equation}
\begin{equation}\label{fif4}\sup_{|z|\leq 1.075}|Q_3'(\sin z)|\leq22.3.\end{equation}
\begin{equation}\label{fif5}\sup_{|z|\leq 1.075}|Q_5'(\sin z)|\leq8.21.\end{equation}

\section{Sixth Order Perturbation Bound}\label{secphi}

In this section, we record an explicit form of the error term of the expansion \eqref{eq:theorem44-expansion} both in the $d=5$ and $d=3$ cases, culminating in \eqref{phi5bd} and \eqref{phi3bd}.  Define
\begin{equation}\label{phidef}
        \varphi_d(\eta)
        \colonequals
        4\pi\frac{H_\eta^{(d,\sigma)}(i)}{i},
\end{equation}
where \(d=5,\sigma=1\) denotes the symmetric \(h_5,h_5\) perturbation of the sign function and
\(d=3,\sigma=-1\) denotes the asymmetric \(h_3,-h_3\) perturbation.  Thus
\[
        A\colonequals h_d(U),\qquad B\colonequals \sigma h_d(V),
\]
with
\[
        h_3(x)\colonequals\frac{2x^3-3x}{\sqrt3\,\pi^{1/4}},
\qquad\text{and}\qquad
        h_5(x)\colonequals\frac{4x^5-20x^3+15x}{2\pi^{1/4}\sqrt{15}}.
\]

At \(z=i\), write, for all $a,b\in\R$,
\[
        p_i(a,b)
        \stackrel{\eqref{eq:pz}}{=}
        \frac{1}{\pi\sqrt2}
        \exp\left(-\frac{a^2+b^2}{2}+iab\right).
\]
\[
        G_i(a,b)
        \stackrel{\eqref{gzdef}}{=}
        \mathbb E\bigl[\operatorname{sign}(X-a)\operatorname{sign}(Y-b)\bigr],
\]
where the expectation is with respect to the complex Gaussian density \(p_i\).
$\forall$ \(0\le j\le 6\), let
\[
        m_j
        \colonequals
        \sup_{a,b\in\mathbb R}
        \left|
        \partial_a^{6-j}\partial_b^jG_i(a,b)
        \right|.
\]
As in Section \ref{app1}, we can compute derivatives of $G$ as follows:
The mixed derivatives are computed from
\[
        \partial_a^m\partial_b^nG_i(a,b)
        =
        4\,\partial_a^{m-1}\partial_b^{n-1}p_i(a,b),
        \qquad m,n\ge1.
\]
Thus, for \(1\le j\le5\),
\[
        \partial_a^{6-j}\partial_b^jG_i(a,b)
        =
        4\,\partial_a^{5-j}\partial_b^{j-1}p_i(a,b).
\]
For the two pure derivatives we use
\[
        \partial_aG_i(a,b)
        =
        -\frac{2e^{-a^2}}{\sqrt\pi}
        \operatorname{erf}\left(\frac{ia-b}{\sqrt2}\right),
\qquad
        \partial_bG_i(a,b)
        =
        -\frac{2e^{-b^2}}{\sqrt\pi}
        \operatorname{erf}\left(\frac{ib-a}{\sqrt2}\right).
\]
Therefore
\[
        \partial_a^6G_i(a,b)
        =
        \partial_a^5
        \left[
        -\frac{2e^{-a^2}}{\sqrt\pi}
        \operatorname{erf}\left(\frac{ia-b}{\sqrt2}\right)
        \right],
\qquad
        \partial_b^6G_i(a,b)
        =
        \partial_b^5
        \left[
        -\frac{2e^{-b^2}}{\sqrt\pi}
        \operatorname{erf}\left(\frac{ib-a}{\sqrt2}\right)
        \right].
\]

A numerical maximization of these explicit functions (for details see Link \ref{oururl}) gives
\[
\begin{array}{c|c|c}
        j & \text{numerical value of }m_j & \text{bound used} \\ \hline
        0 & 38.713601595\ldots & 39 \\
        1 & 6.306143426\ldots  & 6.4 \\
        2 & 2.700948949\ldots  & 2.8 \\
        3 & 1.731970161\ldots  & 1.8 \\
        4 & 2.700948949\ldots  & 2.8 \\
        5 & 6.306143426\ldots  & 6.4 \\
        6 & 38.713601595\ldots & 39 .
\end{array}
\]
Thus we shall use
\begin{equation}\label{mbd}
        (m_0,m_1,m_2,m_3,m_4,m_5,m_6)
        \le
        (39,6.4,2.8,1.8,2.8,6.4,39).
\end{equation}

Next define the one-dimensional standard-normal moments
\[
        N_{d,k}
        \colonequals
        \mathbb E_{Z\sim N(0,1)} |h_d(Z)|^k.
\]
The numerical values, rounded upward, are
\[
\begin{array}{c|ccccccc}
        k & 0&1&2&3&4&5&6 \\ \hline
        N_{3,k}
        &1&0.962&6.207&112.21&3370&1.466\cdot10^5&8.575\cdot10^6\\
        N_{5,k}
        &1&1.04&42.74&9680&4.952\cdot10^6&4.689\cdot10^9&7.289\cdot10^{12}.
\end{array}
\]

We now bound the sixth derivative of \(\varphi_d\) from \eqref{phidef}.  Since
\[
        H_\eta^{(d,\sigma)}(i)
        \stackrel{\eqref{one8}}{=}
        \frac{\pi}{2}
        \iint_{\mathbb R^2}
        G_i(\eta A,\eta B)p_i(u,v)\,du\,dv,
\]
we have
\[
        \varphi_d(\eta)
        \stackrel{\eqref{phidef}}{=}
        \frac{2\pi^2}{i}
        \iint_{\mathbb R^2}
        G_i(\eta A,\eta B)p_i(u,v)\,du\,dv.
\]
Also, for all $u,v\in\R$
\[
        |p_i(u,v)|
        =
        \frac{1}{\pi\sqrt2}e^{-(u^2+v^2)/2}
        =
        \sqrt2\,\gamma(u)\gamma(v),
\]
where \(\gamma\) is the standard normal density.  Hence
\[
\begin{aligned}
        \left|\varphi_d^{(6)}(\eta)\right|
        &\le
        2\pi^2
        \sum_{j=0}^{6}
        \binom6j
        m_j
        \iint_{\mathbb R^2}
        |h_d(u)|^{6-j}|h_d(v)|^j
        |p_i(u,v)|\,du\,dv  \\
        &=
        2\pi^2\sqrt2
        \sum_{j=0}^{6}
        \binom6j
        m_j
        N_{d,6-j}N_{d,j}.
\end{aligned}
\]
Therefore
\begin{equation}\label{phidbd}
        \sup_{\eta\in\mathbb R}
        \left|\varphi_d^{(6)}(\eta)\right|
        \le
        2\pi^2\sqrt2
        \sum_{j=0}^{6}
        \binom6j
        m_j
        N_{d,6-j}N_{d,j}.
\end{equation}

By Taylor's theorem with integral remainder,
\[
        \varphi_d(\eta)
        =
        \varphi_d(0)
        +
        \frac{\varphi_d^{(4)}(0)}{24}\eta^4
        +
        \frac{\eta^6}{5!}
        \int_0^1(1-t)^5\varphi_d^{(6)}(t\eta)\,dt,
\]
because the odd derivatives vanish and the second derivative vanishes for
the admissible perturbations.  Hence
\begin{equation}\label{phibd2}
        \left|
        \varphi_d(\eta)
        -
        \varphi_d(0)
        -
        \frac{\varphi_d^{(4)}(0)}{24}\eta^4
        \right|
        \stackrel{\eqref{phidbd}}{\le}
        C_d|\eta|^6,
\end{equation}
where
\[
        C_d
        \colonequals
        \frac{\pi^2\sqrt2}{360}
        \sum_{j=0}^{6}
        \binom6j
        m_j
        N_{d,6-j}N_{d,j}.
\]
Using the displayed numerical bounds \eqref{mbd} gives
\begin{equation}\label{cbd}
        C_3 < 2.7\cdot10^7,
        \qquad
        C_5 < 2.3\cdot10^{13}.
\end{equation}
% Indeed, the corresponding sums satisfy
% \[
%         \sum_{j=0}^{6}
%         \binom6j
%         m_j
%         M_{3,6-j}M_{3,j}
%         <
%         6.82\cdot10^8,
% \]
% and
% \[
%         \sum_{j=0}^{6}
%         \binom6j
%         m_j
%         M_{5,6-j}M_{5,j}
%         <
%         5.69\cdot10^{14}.
% \]

For the $d=5$ case, \cite{braverman13} compute
\[
        \frac{\varphi_5^{(4)}(0)}{24}
        =
        1600\sqrt2.
\]
Thus, for \(|\eta|\le10^{-2}\),
\begin{equation}\label{phi5bd}
        \boxed{
        \left|
        \varphi_5(\eta)-\varphi_5(0)-1600\sqrt2\,\eta^4
        \right|
        \stackrel{\eqref{phibd2}\wedge\eqref{cbd}}{\le}
        2.3\cdot10^{13}|\eta|^6 .
        }
\end{equation}
The same estimate in fact holds for all real \(\eta\), since the bound above
uses a global supremum over \(a,b\).  For the $d=3$ case corresponding to the \(f_{\eta},f_{-\eta}\) perturbation,
\[
        \frac{\varphi_3^{(4)}(0)}{24}
        =
        48\sqrt2.
\]
Therefore, for \(|\eta|\le10^{-2}\),
\begin{equation}\label{phi3bd}
        \boxed{
        \left|
        \varphi_3(\eta)-\varphi_3(0)-48\sqrt2\,\eta^4
        \right|\stackrel{\eqref{phibd2}\wedge\eqref{cbd}}{\le}
        2.7\cdot10^7|\eta|^6 .
        }
\end{equation}

\section{Numerical optimization of the planar K\texorpdfstring{\"{o}}{o}nig functional}\label{sectiger}

In this section, we provide more detail for Answer \ref{answer1}.  That is, we provide numerical evidence that the planar optimizer of K\"{o}nig's bilinear functional is not an alternating Krivine rounding scheme.  K\"{o}nig's bilinear functional is
\begin{equation}\label{bkdef}
B_K(f,g)
\colonequals
\int_{\mathbb R^2}\int_{\mathbb R^2}
f(x)g(y)e^{-(\|x\|_2^2+\|y\|_2^2)/2}\sin\langle x,y\rangle\,dx\,dy,
\end{equation}
defined for measurable odd $\pm1$-valued functions
\[
f,g\colon\mathbb R^2\to \{-1,1\}.
\]

\subsection{Discretization and coordinate ascent}

Let $x,y\in\R^{2}$.  Write $x=(r\cos\theta,r\sin\theta)$ and $y=(s\cos\phi,s\sin\phi)$ in polar coordinates.  
For the radial integrals we use the substitution
\[
u=\frac{r^2}{2},\qquad v=\frac{s^2}{2},
\]
so that the Gaussian factor becomes $e^{-u}e^{-v}$ and the radial integration is naturally approximated by Gauss--Laguerre quadrature.  
For the angular variables we use the trapezoidal rule on a uniform grid.  Thus $f$ and $g$ are represented by sign arrays on a polar grid with variables
\[
(r_i,\theta_a),\qquad (s_j,\phi_b),
\]
and the discretized integrated kernel is
\[
K_{(i,a),(j,b)}
=
w_i w_j \,(\theta_a - \theta_{a-1})\,(\phi_{b}-\phi_{b-1})\,
\sin\bigl(r_i s_j\cos(\theta_a-\phi_b)\bigr),
\]
where $w_i,w_j$ are the Gauss--Laguerre weights (including the Jacobian factors).

We then perform alternating best-response updates.  For any $m\geq0$, define
\[
g^{(m+1)}\colonequals\operatorname{sign}(T f^{(m)}),\qquad
f^{(m+1)}\colonequals\operatorname{sign}(T g^{(m+1)}),
\]
where, as in \cite[Section 3]{braverman13} we have
\[
(Tf)(y)\colonequals\int_{\mathbb R^2} f(x)e^{-\|x\|_2^2/2}\sin\langle x,y\rangle\,dx.
\]
Each update is optimal with the other function held fixed, so this is a coordinate-ascent procedure for the discretized $B_K$ functional, as noted in \cite{braverman13}.

\subsection{Best refined run}

Our best refined candidate used an optimization grid of size
\[
160\times 4096
\]
in radial versus angular coordinates.  The resulting non-hyperplane fixed point $(f,g)$ resembles the tiger partition from \cite{braverman13} and satisfies the discrete fixed-point equations exactly on the grid:
\[
f=\operatorname{sign}(Tg),\qquad g=\operatorname{sign}(Tf),
\]
with zero sign mismatches.  The optimized value was
\[
B_K(f,g)\approx 11.094573248978.
\]
For comparison, two sign functions in the plane have value
\[
B_K(\operatorname{sign}(x_1),\operatorname{sign}(x_1))
\stackrel{\eqref{bkdef}}{=}
4\pi\,\operatorname{arsinh}(1)
\approx 11.075667144195.
\]
Hence the numerical tiger candidate improves the hyperplane benchmark by $\approx 0.0189061047.
$

\subsection{Pictures of the optimized functions \texorpdfstring{$f$}{f} and \texorpdfstring{$g$}{g}}

Figure~\ref{fig:tiger-f} shows the optimized function $f$, and
Figure~\ref{fig:tiger-g} shows the corresponding optimized function $g$.

\begin{figure}[ht]
\centering
\includegraphics[width=0.78\textwidth]{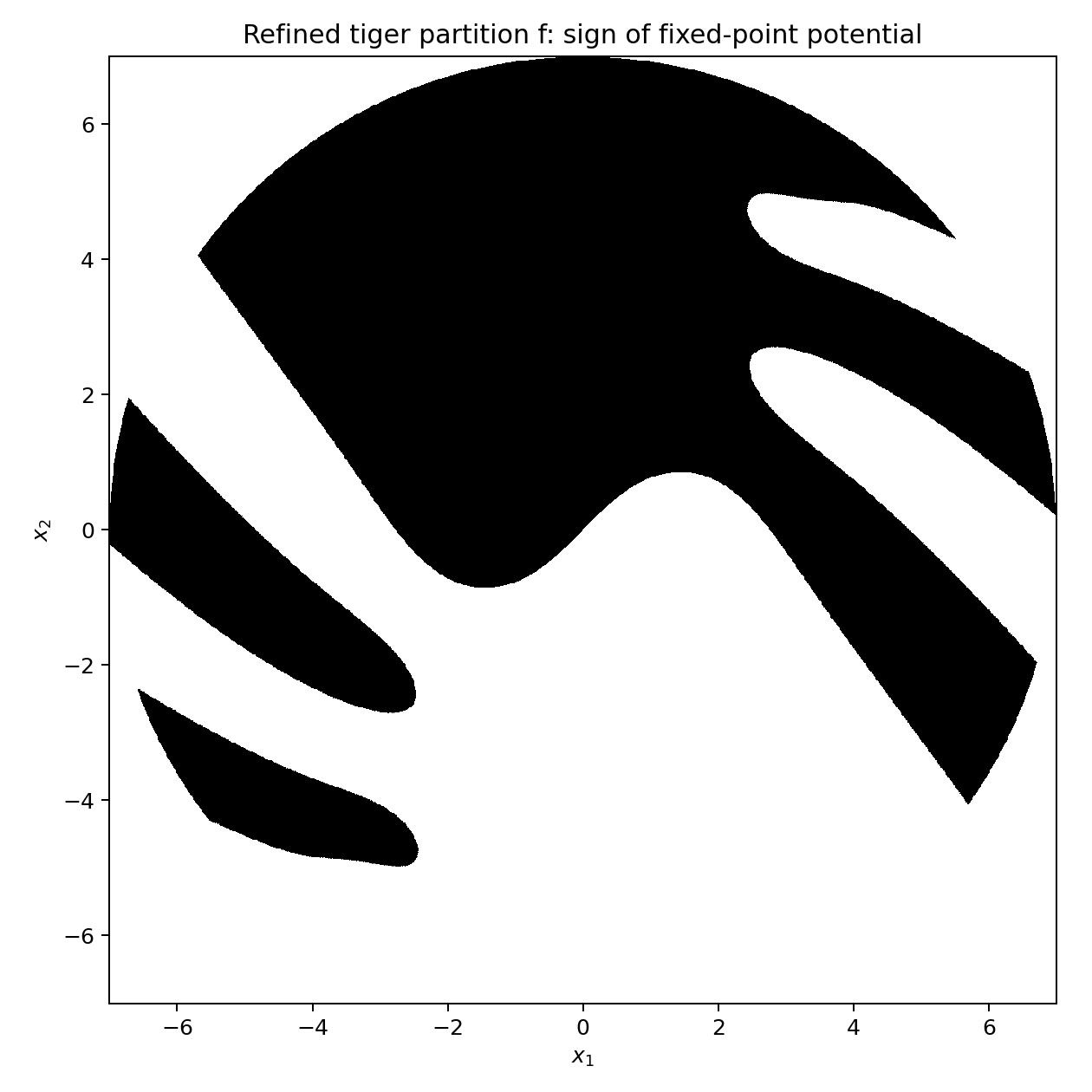}
\caption{Refined numerical tiger partition $f$.  The black region is $\{x:f(x)=1\}$ and the white region is $\{x:f(x)=-1\}$.}
\label{fig:tiger-f}
\end{figure}

\begin{figure}[ht]
\centering
\includegraphics[width=0.78\textwidth]{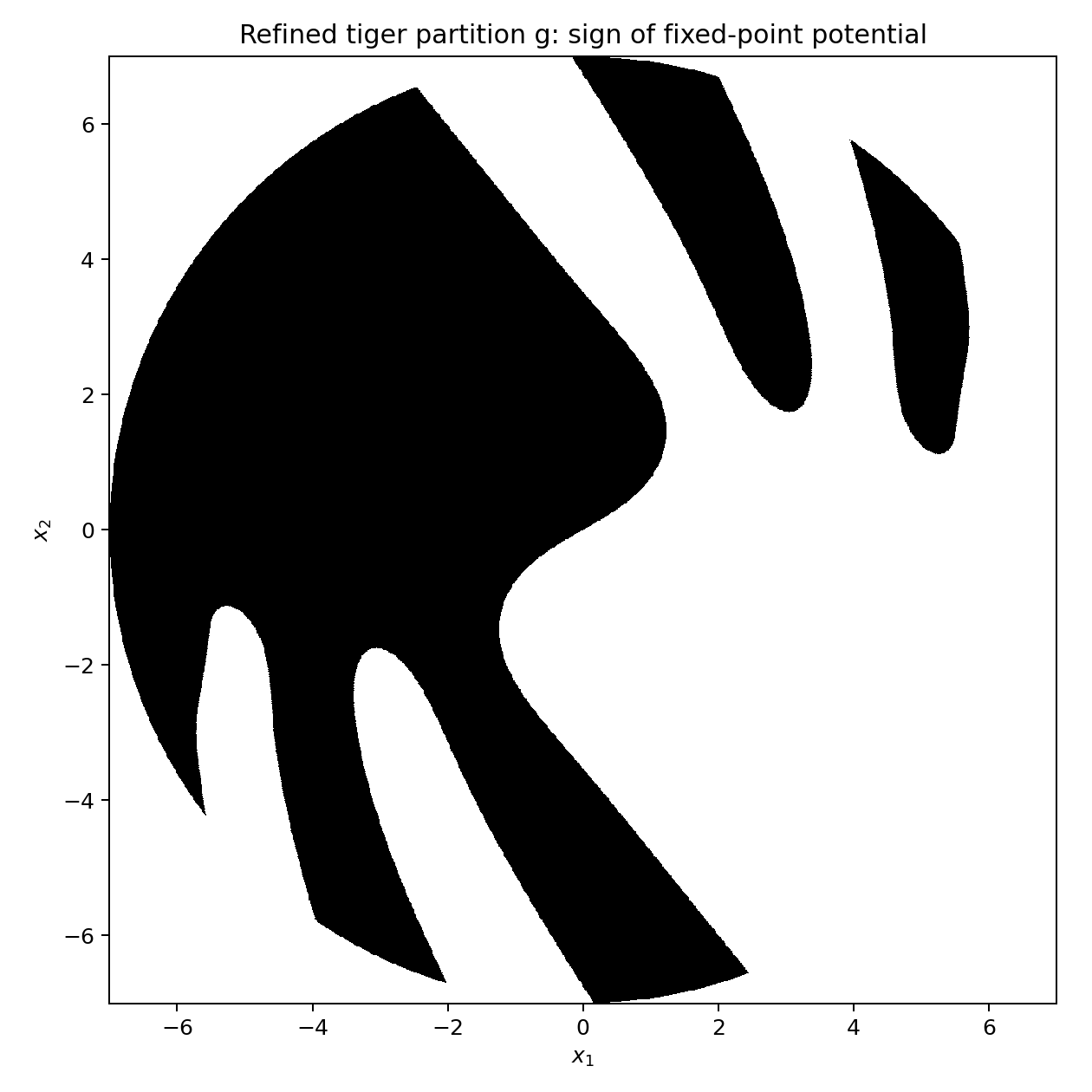}
\caption{Refined numerical tiger partition $g$.  The black region is $\{x:g(x)=1\}$ and the white region is $\{x:g(x)=-1\}$.}
\label{fig:tiger-g}
\end{figure}

\subsection{Coefficients of \texorpdfstring{$H$}{H} and \texorpdfstring{$H^{-1}$}{H inverse}}

Let $(f,g)$ denote the optimized pair from Figures \ref{fig:tiger-f} and \ref{fig:tiger-g}.  As in \eqref{one8}, \eqref{one9} and \eqref{one10}, we examine the Taylor coefficients of $H_{f,g}(t)$ and its inverse.

Using a finer coefficient grid of size
$
320\times 16384,
$
we obtained the following coefficients:

\begin{center}
\begin{tabular}{rcc}
\hline
degree & coefficient of $H_{f,g}$ & coefficient of $H_{f,g}^{-1}$ \\
\hline
$1$  & $0.6271552605725179$   & $1.594501494075198$ \\
$3$  & $0.09456848251227919$  & $-0.6112883827335565$ \\
$5$  & $0.04449372752246313$  & $-\textbf{0.028166647067356606}$ \\
$7$  & $0.025141713405143495$ & $0.11401753102956447$ \\
$9$  & $0.017053623803116084$ & $-0.12414815588752107$ \\
$11$ & $0.012245521324230563$ & $0.13002877790599687$ \\
$13$ & $0.00923026485410719$  & $-0.05335404389331977$ \\
\bottomrule
\end{tabular}
\end{center}

Thus \eqref{one8} has the approximation for any $t$ near $0$ as
\begin{flalign*}
H_{f,g}(t)
&\approx
0.6271552606\,t
+0.0945684825\,t^3
+0.0444937275\,t^5\\
&\qquad+0.0251417134\,t^7
+0.0170536238\,t^9
+\cdots
\end{flalign*}
and \eqref{one9} has the approximation for any $t$ near $0$ as
\begin{flalign*}
H_{f,g}^{-1}(z)
&\approx
1.5945014941\,z
-0.6112883827\,z^3
-\textbf{0.0281666471}\,z^5\\
&\qquad+0.1140175310\,z^7
-0.1241481559\,z^9
+\cdots.
\end{flalign*}

In particular, the coefficient of $z^5$ in $H_{f,g}^{-1}$ is negative:
\[
\widehat a_5\approx -0.0281666471.
\]
Since an alternating Krivine scheme would require
\[
\operatorname{sign}(\widehat a_{2j+1})=(-1)^j,
\]
the sign of $\widehat a_5$ violates the definition of alternating Krivine rounding scheme.  That is, it appears that the answer to Question \ref{altq} is: No.

\subsection{More detail on the fifth coefficient}

Denote the coefficients of $H_{f,g}$ as
\[
H_{f,g}(t)\colonequals\widehat\beta_1 t+\widehat\beta_3 t^3+\widehat\beta_5 t^5+\cdots.
\]
Then series reversion gives
%A_5=a_1^(-9)(3a_1^2a_3^2-a_1^3a_5)
% = a_1^(-7)(3a_3^2 - a_1 a_5)
\[
a_5=\frac{3\widehat\beta_3^2-\widehat\beta_1\widehat\beta_5}{\beta_1^7}.
\]
Using the numerical values above,
\[
3\widehat\beta_3^2-\widehat\beta_1\widehat\beta_5
\approx
-0.00107488162417<0,
\]
hence
\[
\widehat a_5<0.
\]
% 3*0.09456848251227919^2 - 0.6271552605725179*0.04449372752246313

\section{Explicit Degree 3 Hermite Perturbation}\label{seclast}

In this section we give a rigorous computer-assisted proof using Sage that
\begin{equation}\label{kg5}
K_G<\frac{\pi}{2\log(1+\sqrt{2})}-10^{-5},
\end{equation}
using interval arithmetic.  That is, we prove Theorem \ref{thm2}.  Define
\[
        \eta\colonequals0.04249900400783211,
\]
and define $f,g\colon\R^{2}\to\{-1,1\}$ by
\begin{equation}\label{fgdef}
        f(x_1,x_2)=\operatorname{sign}(x_2-\eta h_3(x_1)),
        \qquad
        g(x_1,x_2)=\operatorname{sign}(x_2+\eta h_3(x_1))
        ,\qquad\forall\,(x_{1},x_{2})\in\R^2.
\end{equation}
Let \(H_{f,g}\) be defined by \eqref{one8}.  As in \eqref{eq:normalization-factor} and \eqref{eq:paper-H}, define
$$H_{\eta}(z)
=\frac{\pi}{2}H_{f,g}(z),\qquad\forall\,|\Re(z)|<1.$$ 
As in \eqref{eq:Bq-coefficients} and \eqref{eq:Aq-def}, define
\begin{equation}\label{bdef}
        B(z)\colonequals H_{\eta}^{-1}(z)\equalscolon\sum_{n\ge1}a_nz^n
\end{equation}
\[
        A(t)\colonequals\sum_{n\ge1}|a_n|t^n.
\]
Recall from Section \ref{secquant} and \eqref{cpi} that if $\gamma>0$ satisfies $A(\gamma)=1$, and if \(A(L)<1<A(0.9)\), then $\gamma$ lies in \((L,0.9)\), and the corresponding Grothendieck bound is
\begin{equation}\label{kggbd}
        \KG\le \frac{\pi}{2\gamma}.
\end{equation}

% I=int_R mu, I^2 = 1/pi *2pi * \int re^-r^2
%        =1
\subsection{Exact coefficient formula used in the certificate}  
% Denote 
% $\mu(dx)\colonequals\pi^{-1/2}e^{-x^2}\,dx$.  Note that $\int_{\R}\mu(x)dx=1$.  
$\forall$ $m\geq0$, denote
\begin{equation}\label{herdef}
h_{m}(x)\colonequals\frac{(-1)^{m}}{\sqrt{2^{m}m!\sqrt{\pi}}}e^{x^{2}}\Big(\frac{d}{dx}\Big)^{m}e^{-x^{2}},
\qquad
        \psi_m(x)\colonequals\pi^{1/4}h_m(x),\quad\forall\,x\in\R.
\end{equation}
Then \((\psi_m)_{m\ge0}\) is orthonormal in \(L_2(\mu)\), where
\begin{equation}\label{mudef}
d\mu(x)=\pi^{-1/2}e^{-x^2}\,dx.
\end{equation}
Note that $\int_{\R}d\mu(x)=1$.  Also $(\Psi_{m,n})_{m,n\geq0}$ is an orthonormal basis of \(L_2(\mu\otimes\mu)\), where
\[
\Psi_{m,n}(x_1,x_2)\colonequals\psi_m(x_1)\psi_n(x_2),\qquad\forall\,m,n\geq0
\]

For any $n\geq0$, $w\in\R$, define
\begin{equation}\label{indef}
I_n(w)\colonequals\int_{\R}\operatorname{sign}(x-w)\psi_n(x)\,d\mu(x).
\end{equation}
% (1/\pi^1/2)\int_a^\infty e^-x^2
%  - \int_-\infty^a e^-x^2
% = -2\int_0^a e^-x^2 / pi^1/2
%
% 
Then 
\[
        I_0(w)=-\erf(w),
\]
and integration by parts with \eqref{herdef} shows, for \(n\ge1\) and $w\in\R$,
\begin{equation}\label{ineq}
        I_n(w)=\sqrt{\frac{2}{\pi n}}e^{-w^2}\psi_{n-1}(w).
\end{equation}
%\int e^-x^2 sign(x-w) e^x^2 (d/dx)^m e^-x^2
%=\int[sign(x-w)]'(d/dx)^m-1 e^-x^2
%=-2(d/dx)^m-1 e^-x^2 |_x=w
%= 2e^-w^2 h_m-1   -cm / cm-1  pi^-1/2
%= 2/sqrt(2m)   pi^-1/2

For the two sign functions $f,g$, define their Hermite coefficients by
\begin{flalign*}
\widehat{f}(m,n)
&\colonequals
\int_{\mathbb R^2} f(x_1,x_2)\Psi_{m,n}(x_1,x_2)\,d\mu(x_1)d\mu(x_2)\\
&\stackrel{\eqref{fgdef}}{=}\int_{\R^{2}}\mathrm{sign}(x_{2}-\eta h_3 (x_1))\psi_m(x_1)\psi_n (x_2)d\mu(x_{1})d\mu(x_2)\\
&\stackrel{\eqref{indef}}{=}\int_{\R}\psi_m(x_1)I_n(\eta h_3(x_1))d\mu(x_1), \qquad\qquad\forall\,m,n\geq0,
\end{flalign*}
where \eqref{ineq} allows $\widehat{f}(m,n)$ to be written as a one-dimensional integral, which is a much simpler task to estimate on a computer than a two-dimensional integral.  Similarly,
\[
\widehat{g}(m,n)
\colonequals
\int_{\mathbb R^2} g(y_1,y_2)\Psi_{m,n}(y_1,y_2)\,d\mu(y_1)d\mu(y_2)
=\int_{\R}\psi_m(x_1)I_n(-\eta h_3(x_1))d\mu(x_1).
\]
Thus, in \(L_2(\mu\otimes\mu)\),
\begin{equation}\label{fgl2}
f=\sum_{m,n\ge0}\widehat{f}(m,n)\Psi_{m,n},
\qquad
g=\sum_{m,n\ge0}\widehat{g}(m,n)\Psi_{m,n}.
\end{equation}

Define
\begin{equation}\label{ckdef}
        \beta_k\colonequals\frac{\pi}{2}\sum_{m=0}^k\widehat{f}(m,k-m)\widehat{g}(m,k-m).
\end{equation}
Then the Taylor coefficients of \(H_\eta\) are
\begin{equation}\label{htaylor}
\begin{aligned}
H_\eta(z)
&\stackrel{\eqref{eq:normalization-factor}\wedge\eqref{eq:paper-H}\wedge\eqref{eq:pz}}{=}
\frac{\pi}{2}\int_{\R^{2}\times\R^{2}}f(x)g(y)p_{z}(x_1,y_1)p_{z}(x_2,y_2)dxdy\\
&\stackrel{\eqref{fgl2}\wedge\eqref{eq:hermite-covariance}\wedge\eqref{herdef}}{=}
\frac{\pi}{2}\sum_{m,n\geq0}\widehat{f}(m,n)\widehat{g}(m,n)z^{m+n}
\stackrel{\eqref{ckdef}}{=}\sum_{k\ge1}\beta_kz^k,
\end{aligned}
\end{equation}
% int psi_m psi_n p_z
% = sqrt(pi)\int h_m h_n p_z
The inverse coefficients from \eqref{bdef} are computed by the finite recursion
\[
        a_1=\frac1{\beta_1},
\]
and for \(n\ge2\),
\begin{equation}\label{aneq}
        a_n=-\frac1{\beta_1}\sum_{k=2}^n \beta_k[z^n]
        \left(\sum_{j=1}^{n-1}a_jz^j\right)^k.
\end{equation}
%Our Sage code encloses all integrals and the inverse recursion by ball arithmetic.

% H = C1z + C2z^2 + ...
%an = n^th coefficient of H^-1
% n a_n = [x^n-1](x/H)^n

Indeed, for any $n\geq2$, 
\begin{flalign*}
0
&=[z^n] H_{\eta}(H_{\eta}^{-1}(z))
\stackrel{\eqref{htaylor}}{=}
\sum_{k\geq1}[z^n]\beta_k (H_{\eta}^{-1}(z))^k
\stackrel{\eqref{bdef}}{=}\beta_1 a_n + \sum_{k=2}^{n}[z^n]\beta_k (H_{\eta}^{-1}(z))^k\\
&\stackrel{\eqref{bdef}}{=}\beta_1 a_n + \sum_{k=2}^{n}[z^n]\beta_k \Big(\sum_{j=1}^{n-1}a_j z^j\Big)^k,
\end{flalign*}
since the terms $j\geq n$ have no contribution to the $[z^n]$ coefficient in the sum where $k\geq2$.  Solving for $a_n$ proves \eqref{aneq} which gives an algorithm for computing the coefficients of $H_{\eta}^{-1}$ when given the coefficients $\beta_k$ of $H_\eta$ from \eqref{htaylor}.

\subsection{The inverse-tail/Rouch\texorpdfstring{\'{e}}{e} condition}
The Sage code gives rigorous enclosures for
\[
        \sum_{n=1}^{201}|a_n|L^n.
\]
To pass from this finite sum to the full infinite sum
\begin{equation}\label{alsum}
        A(L)=\sum_{n\ge1}|a_n|L^n,
\end{equation}
one must also justify that \(B=H_\eta^{-1}\) is analytic on a disk of radius strictly larger than \(L\), and one needs a bound on \(B\) on that disk.  The latter is supplied by a Rouch\'{e}-type inverse-domain check.  It is convenient to work in the \(\zeta\)-coordinate and therefore define
\begin{equation}\label{calhdef}
        \mathcal H_\eta(\zeta)\colonequals H_\eta(\sin\zeta),\qquad\forall\,\zeta\in\C.
\end{equation}
% Since \(H_0(\sin\zeta)=\zeta\), \(\mathcal H_\eta\) is a small perturbation of the identity.  
%Let
% \[
%         0<r<s<\pi/2.
% \]
As we will describe below, the Sage code will verify the following for some $0<r<s<\pi/2$:
\begin{equation}\label{eq:rouche-condition}
        \sup_{|\zeta|=s}|\mathcal H_\eta(\zeta)-\zeta|<s-r.
\end{equation}
In particular, we will use $s=1.05$, $r=1.05-.08113$.  The following lemma then uses \eqref{eq:rouche-condition} to prove a tail bound.

\begin{lemma}[Rouch\'{e} inverse-tail certificate]\label{lem:rouche-tail}
Let $0<r<s<\pi/2$.  Assume \eqref{eq:rouche-condition} holds.  Then \(B=H_\eta^{-1}\) is holomorphic on \(|w|<r\).  Moreover,
\[
        \sup_{|w|\le r}|B (w)|\le \sinh s.
\]
Consequently, for every \(N\ge1\),
%\[
%        B_\eta(w)=\sum_{n\ge1}a_nw^n,
%\]
%then, 
\[
        \sum_{n>N}|a_n|L^n
        \le
        \sinh(s)\frac{(L/r)^{N+1}}{1-L/r}.
\]
% In particular, since \(\sinh(1.10)<2\), we get
% \[
%         \sum_{n>N}|a_n|L^n
%         \le
%         2\frac{(L/r)^{N+1}}{1-L/r}
% \]
% holds whenever \eqref{eq:rouche-condition} holds for any \(r<s=1.1\).
\end{lemma}

\begin{proof}
Fix $w\in\C$ with \(|w|\le r\).  Suppose $|\zeta|=s$.  By the triangle inequality
\[
        |\zeta-w|\ge s-r.
\]
By \eqref{eq:rouche-condition},
\[
        |\mathcal H_\eta(\zeta)-\zeta|<s-r\le |\zeta-w|.
\]
Hence the function \(\zeta\mapsto\mathcal H_\eta(\zeta)-w=H_\eta(\zeta)-\zeta+\zeta-w\) and the function \(\zeta\mapsto\zeta-w\) have the same number of zeros in \(|\zeta|<s\), namely one, by Rouch\'{e}'s theorem.  Thus \(\mathcal H_\eta^{-1}\) is holomorphic on \(|w|<r\), with values in \(|\zeta|<s\).  Therefore
\[
        B (w)\stackrel{\eqref{bdef}\wedge\eqref{calhdef}}{=}\sin(\mathcal H_\eta^{-1}(w))
\]
% B = H^-1
% H^-1 = ? sin (cal H^-1)
% cal H = H o sin
% cal H^-1 = sin^-1 H^-1
is holomorphic on \(|w|<r\), and
\[
        |B (w)|\le \sup_{|\zeta|\le s}|\sin\zeta|
        \le \sinh s.
\]
Cauchy's estimate gives \(|a_n|\le \sinh(s)r^{-n}\).  Summing the geometric tail proves the claim.
\end{proof}

The numerical output quoted below will use
\[
        r=1.05 - \delta,
        \qquad
        s=1.05,
\]
where $\delta\colonequals0.08113$,
%\colonequals0.08112042094580829143524169921875\leq.08113$
so by Lemma \ref{lem:rouche-tail} with $N=201$, we get
\begin{equation}\label{atail}
        \sum_{n>N}|a_n|L^n
        \leq\sinh(1.05)\frac{(L/(1.05-\delta))^{202}}{1-L/(1.05-\delta)}
        \leq7\cdot10^{-8}.
\end{equation}
%delta=0.08113, L=log(1+sqrt(2)), sinh(1.05)*(L/(1.05-delta))^(202) / (1 - L/(1.05-delta))
This tail is valid since the inequality
\[
        \sup_{|\zeta|=1.05}|\mathcal H_\eta(\zeta)-\zeta|<\delta
\]
is verified by our Sage code \verb!rouche_v9.sage!, which took about 54 hours to run. Since the tail bound \eqref{atail} is small, it remains to accurately estimate the finite sum to index $201$.

\subsection{Certifying \texorpdfstring{$A(L)<1$}{A(L)<1}}
The successful Sage run of \verb!groth_series_verify! gave the following rigorous ball enclosures for the finite part of the sum:
\[
\begin{aligned}
        \sum_{n=1}^{201}|a_n|L^n
        &\le
        0.9999926809247332727295843168961274199283134071967509
        +4.27\cdot10^{-58},
\end{aligned}
\]
with $(a_n)$ computed via \eqref{aneq} and \eqref{ckdef}.  Combining this estimate with \eqref{alsum} and \eqref{atail} gives
$$
        A(L)
        \le        0.99999268092473327272958431
        +7\cdot10^{-8}
        \le 0.999992750924734.
$$
In particular,
\begin{equation}\label{ten1}
        A(L)<1-7.249\cdot10^{-6}.
\end{equation}
% Even with the more conservative \(r=1.00\) tail one still obtains
% \[
%         A(L)<1-7.3189\cdot10^{-6},
% \]
% so the conclusion is insensitive to the tail choice.

The same run also certified the derivative, using term-by-term differentiation
\[
        \sup_{L\le t\le0.9}A'(t)
        \le
        1.4459221948078752122929593376044317888266688361343087
        +5.71\cdot10^{-58}.
\]
In particular,
\begin{equation}\label{ten2}
        \sup_{L\le t\le0.9}A'(t)<1.446.
\end{equation}

\subsection{Certifying \texorpdfstring{$A(.9)$}{A(.9)} and Sage integration compatibility}
Below we will apply the intermediate value theorem to the function $A$.  We verified $A(L)<1$, so it remains to check
\[
        A(0.9)>1.
\]
Since all terms in \(A(t)\) are nonnegative, it is enough to use only the first two odd inverse coefficients:
\[
        A(0.9)\ge |a_1|(0.9)+|a_3|(0.9)^3.
\]
The included file \texttt{arb\_h3\_A09\_certificate\_panel.sage} computes rigorous ball enclosures for \(a_1\) and \(a_3\), then proves
\[
        |a_1|(0.9)+|a_3|(0.9)^3>1.02.
\]
This proves \(A(0.9)>1\).
 
The Sage code performs a rigorous interval-ball subdivision integral on \([-M,M]\): on each subinterval \(I\), the integrand is evaluated on a real ball enclosing \(I\), so the range of the integrand on \(I\) is enclosed, and multiplying by the interval width encloses the integral over \(I\).  The tails outside \([-M,M]\) are bounded by the Cauchy-Schwarz inequality as
\[
        \int_{|x|>M}|\psi_m(x)I_n(\pm\eta h_3(x))|\,d\mu(x)
        \le \sqrt{\mu(|x|>M)},
\]
(using also $|I_n(w)|\leq1$ for all $w\in\R$ by Cauchy-Schwarz),
and for any $M>0$
\[
        \mu(|x|>M)\stackrel{\eqref{mudef}}{=}\operatorname{erfc}(M)
        \le \frac{e^{-M^2}}{\sqrt\pi M}.
\]
Thus the \(A(0.9)>1\) check is self-contained and does not rely on a Sage-provided integration method.

\subsection{Concluding the \texorpdfstring{$K_G$}{KG} Bound}
\begin{proof}[Proof of Theorem \ref{thm2}] Since all above inequalities are certified by Sage, the inequalities
\[
        A(L)<1<A(0.9)
\]
imply that there exists \(\gamma\in(L,0.9)\) with \(A(\gamma)=1\) by the Intermediate Value Theorem.  Moreover, by the Mean Value Theorem,
\begin{equation}\label{lasteq}
        1-A(L)=A(\gamma)-A(L)
        \le
        \left(\sup_{L\le t\le0.9}A'(t)\right)(\gamma-L).
\end{equation}
Using the certified values above from \eqref{ten1} and \eqref{ten2}
\[
        1-A(L)
        \ge
        7.249\cdot10^{-6},
\qquad
        \sup_{L\le t\le0.9}A'(t)
        \le
        1.446.
\]
Thus
\[
        \gamma-L
        \stackrel{\eqref{lasteq}}{\ge}
        5.013\cdot10^{-6}.
\]
%(7.249* 10^(-6))/1.446
%Therefore
\[
        \KG
        \stackrel{\eqref{kggbd}}{\le}
        \frac{\pi}{2\gamma}
        \le
        \frac{\pi}{2(L+5.013\cdot10^{-6})}.
\]
% Numerically,
% \[
%         %\boxed{
%         \KG<1.782203606913.
%         %}
% \]
That is,
\[
        \boxed{
        \KG<\frac{\pi}{2\log(1+\sqrt2)}-1.013\cdot10^{-5}.
        }
\]
% L=log(1+sqrt(2)), pi/(2*L) - pi/(2*(L+5.013*10^(-6))) 

\end{proof}

\noindent\textbf{Acknowledgement.}  ChatGPT 5.5 assisted in the preparation of this manuscript.

\bibliographystyle{amsalpha}

\end{document}